\documentclass[11pt,a4paper,reqno]{amsart}

\usepackage[a4paper,margin=0.95in]{geometry}
\usepackage{amsmath,amssymb,amsthm,mathtools,mathrsfs,bbm}
\usepackage{xcolor}

\usepackage[expansion=false]{microtype}
\usepackage[colorlinks=true,linkcolor=blue,citecolor=red,urlcolor=blue]{hyperref}
\hypersetup{
  pdftitle={Semiclassical Moser--Trudinger inequalities for orthonormal families},
  pdfauthor={Phuoc-Tai Nguyen and Tuan-Dat Tran},
  pdfsubject={Endpoint inequalities and semiclassical density limits}
}
\allowdisplaybreaks
\newtheorem{theorem}{Theorem}[section]
\newtheorem{lemma}[theorem]{Lemma}
\newtheorem{proposition}[theorem]{Proposition}
\newtheorem{corollary}[theorem]{Corollary}
\theoremstyle{remark}
\newtheorem{remark}[theorem]{Remark}

\newcommand{\R}{\mathbb R}
\newcommand{\N}{\mathbb N}
\newcommand{\cA}{\mathcal{A}}
\newcommand{\cE}{\mathcal E}
\newcommand{\cQ}{\mathcal{Q}}
\newcommand{\cT}{\mathcal{T}}
\newcommand{\scrH}{\mathscr{H}}
\newcommand{\scrQ}{\mathscr{Q}}
\newcommand{\Ei}{\operatorname{E_1}}

\newcommand{\1}{\mathbbm{1}}

\title[Semiclassical Moser--Trudinger inequalities ]{Semiclassical Moser--Trudinger type inequalities \\ for orthonormal families in $\mathbb{R}^2$}

\author[P.-T. Nguyen]{Phuoc-Tai Nguyen}
\address{Department of Mathematics and Statistics, Masaryk University, Kotl\'a\v rsk\'a 2, 61137 Brno, Czech Republic}
\email{ptnguyen@math.muni.cz}

\author[T. D. Tran]{Tuan-Dat Tran}
\address{Department of Mathematics and Statistics, Masaryk University, Kotl\'a\v rsk\'a 2, 61137 Brno, Czech Republic}
\email{554671@mail.muni.cz}

\numberwithin{equation}{section}

\begin{document}

\begin{abstract}
For $H^1(\R^2)$-orthonormal families $\{u_n\}_{n=1}^N$ with one-body density $\rho=\sum_{n=1}^N |u_n|^2$, we prove the whole-space endpoint Moser--Trudinger inequality
$$ \int_{\R^2}\frac{e^{4\pi\rho(x)}-1-4\pi\rho(x)}{\rho(x)}\,dx \le C\left(N-1+\int_{\R^2}\rho(x)\,dx\right),
$$
where the integrand on the left-hand side is defined to be zero on $\{x: \rho(x)=0\}$.  We develop a  heat-semigroup refinement of Rumin's density method to treat both the endpoint coefficient $4\pi$ and the whole-space setting. Our inequality is sharp in two respects: the optimality of the coefficient $4\pi$ and the linear growth in the rank $N$.  Moreover, through a semiclassical analysis of the  density generated by normalized Dirichlet eigenfunctions, we establish the collective asymptotic behaviors of the quantities appearing in the inequality in various regimes. These semiclassical behaviors show that the rank-dependent term $N-1$ on the right-hand side is indispensable.
 \medskip

 \medskip

\end{abstract}

\maketitle

\setcounter{tocdepth}{1}
\tableofcontents

\section{Introduction}

\subsection{Background and setting}

For a bounded domain $\Omega\subset\R^2$, the Sobolev embedding $H_0^1(\Omega) \hookrightarrow L^p(\Omega)$ holds for every $p \in [1,\infty)$ but fails for $p=\infty$. In two-dimensional settings, the Moser--Trudinger inequality provides a substitute for the missing endpoint Sobolev embedding, yielding exponential integrability in place of an $L^\infty$ bound. The Moser--Trudinger
inequality, derived by Trudinger in
\cite{Trudinger67} and sharpened by Moser \cite{Moser71}, states that 
\begin{equation}  \label{eq:classical-mt}
 \sup_{0 \neq u\in H_0^1(\Omega), \,
                  \|\nabla u\|_{L^2(\Omega)}\leq 1}
 \frac{1}{|\Omega|}\int_\Omega e^{4\pi |u(x)|^2} \,dx < + \infty.
\end{equation}
Moreover, the constant $4\pi$ is optimal in the sense that if $4\pi$ is replaced by any $\alpha>4\pi$, then the supremum in \eqref{eq:classical-mt} becomes infinite.

A natural extension is to establish a counterpart of  \eqref{eq:classical-mt} on the whole space. However, on $\R^2$, the exponential $e^{4\pi |u|^2}$ in \eqref{eq:classical-mt} is non-integrable, even for functions $u$ with compact support.  While this can be remedied by considering the renormalized integrand $e^{4\pi |u|^2}-1$, another obstruction persists under the gradient constraint alone. Specifically, the gradient constraint $\|\nabla u\|_{L^2} \leq 1$ is invariant under the two-dimensional dilation $u_\lambda(x) = u(x/\lambda)$, whereas the integral of the renormalized exponential scales by a factor of $\lambda^2$. Consequently, the gradient constraint alone is insufficient to control the renormalized whole-space integral, and \eqref{eq:classical-mt} has several inequivalent counterparts on $\R^2$.

Adachi and Tanaka~\cite{AdachiTanaka00} proved that for any $0 < \alpha < 4\pi$ and any function $u \in H^1(\R^2)$ with $ \|\nabla u\|_{L^2} \leq 1$, one has
$$
 \int_{\R^2}(e^{\alpha|u|^2}-1)\,dx
 \le C_\alpha\|u\|_{L^2}^2.
$$
Subsequently, Ruf~\cite[Theorem~1.1]{Ruf05} obtained the endpoint inequality with coefficient $4\pi$ by replacing the homogeneous constraint $\| \nabla u \|_{L^2} \leq 1$ with the full Sobolev norm constraint $\| u \|_{H^1} \leq 1$, namely
\begin{equation}  \label{eq:ruf}
 \sup_{\|u\|_{H^1}\le1}
 \int_{\R^2}(e^{4\pi |u|^2}-1)\,dx < + \infty.
\end{equation}
Inequality \eqref{eq:ruf} fails if $4\pi$ is replaced by any larger number. Afterwards, Ibrahim, Masmoudi
and Nakanishi \cite{IMN15} demonstrated that 
\begin{equation} \label{eq:IMN-exact-growth-intro}
 \int_{\R^2}\frac{e^{4\pi|u|^2}-1}{(1+|u|)^2}\,dx
 \le C\|u\|_{L^2}^2 \quad \text{for all } u \in H^1(\R^2) \text{ with } \|\nabla u\|_{L^2} \leq 1.
\end{equation}
They also pointed out that the power $2$ in the denominator of
\eqref{eq:IMN-exact-growth-intro} is optimal, namely it cannot be replaced by any strictly smaller power. In addition, they showed that
\eqref{eq:IMN-exact-growth-intro} implies~\eqref{eq:ruf} by using an argument based on H\"older's inequality. 

A second line of generalization of~\eqref{eq:classical-mt}, motivated by  Lieb--Thirring and semiclassical spectral inequalities, replaces the single function $u$ by
a family of $N$ orthonormal functions $\{u_n\}_{n=1}^N$ 
and seeks bounds on an
exponential-type functional of the one-body density
\begin{equation} \label{eq:one-body-density} \rho(x)=\sum_{n=1}^N|u_n(x)|^2.
\end{equation}
Such semiclassical Moser--Trudinger
inequalities interpolate between the single-function case $N=1$ and Weyl-law statements about the collective behavior of many
eigenfunctions as $N\to\infty$. They serve as a two-dimensional, exponential-type counterpart of the Lieb--Thirring inequality for orthonormal families \cite{Lieb83}. This is in the same spirit as the fact that  Solomyak's
Orlicz-norm inequality~\cite{Solomyak94} is a  two-dimensional
substitute for the Cwikel--Lieb--Rozenblum inequality~\cite{Frank14} which fails in two dimensions. Frank and Laptev~\cite{FrankLaptev18} obtained a
version with constants independent of the Dirichlet domain. By duality,
their result yields constants $a,C>0$ such that, for $N\ge2$ and a
Dirichlet-form orthonormal family on a bounded domain $\Omega$,
$$
 \frac{1}{|\Omega|}\int_\Omega
 \exp\!\left(\frac{a\rho(x)}{\ln N}\right)\,dx\le C.
$$
A broad account of functional inequalities for orthonormal systems is available (see Frank~\cite{Frank22}).  

Arora, Nam and Nguyen~\cite{ANN} sharpened this two-dimensional picture and made the semiclassical constant explicit. They demonstrated that for $N$ functions $\{u_n\} \subset H_0^1(\Omega)$ satisfying
\begin{equation} \label{eq:ann-gradient-orthonormal}
 \int_{\Omega} \nabla u_n \cdot \overline{\nabla u_m}\,dx
 = \delta_{nm}, \quad 1 \leq n,m \leq N,
\end{equation}
and any $\varepsilon \in (0,1/4)$, one has
\begin{equation} \label{eq:ann-subcritical}
 \frac{1}{|\Omega|}\int_\Omega
 e^{4\pi(1-\varepsilon)\rho(x)}\,dx
 \le \frac{e^8}{\varepsilon^4}N.
\end{equation}
Moreover, for any $\alpha>0$, 
\begin{equation} \label{eq:ann-corollary}
\lim_{N \to \infty} \left( \sup_{ \{u_n\}_{n=1}^N } \frac{1}{|\Omega|}\int_\Omega\exp\Big(\alpha\,\frac{\rho}{\ln N}\Big)\,dx \right)
= e^{\frac{\alpha}{4\pi}},
\end{equation}
where the supremum is taken over all  families satisfying \eqref{eq:ann-gradient-orthonormal}. Their method, inspired by
Rumin~\cite{Rumin10}, combines the boundedness of $\Omega$ for
low-momentum control with Bessel's inequality at intermediate momenta
to handle the logarithmic divergence in two dimensions. The constant
$1/(4\pi)$ in \eqref{eq:ann-corollary} is determined by the two-dimensional Weyl law for the
Dirichlet Laplacian. We emphasize that
\eqref{eq:ann-gradient-orthonormal} involves only the gradients of the functions, and estimate \eqref{eq:ann-subcritical} is subcritical and does not reach the endpoint coefficient $4\pi$.

Combining these two independent generalizations to the whole space and to orthonormal families poses significant challenges. The mechanism in~\cite{ANN}, which relies on Poincar\'e's inequality on
bounded domains, fails on $\R^2$, and a
dilation argument shows that no
bound of the form~\eqref{eq:ann-subcritical} can hold on $\R^2$ under
gradient orthonormality~\eqref{eq:ann-gradient-orthonormal}, let alone
an endpoint inequality. In this paper, we address endpoint
Moser--Trudinger inequalities for families orthonormal with respect to
the $H^1$ inner product.

\subsection{Main results}
Let $\cA:=-\Delta+1$ on $L^2(\R^2)$. A family
$\{u_n\}_{n=1}^N \subset H^1(\R^2)$ ($N \geq 1$) is called $H^1$-orthonormal if
$$
 \int_{\R^2}\big(\nabla u_n\cdot\nabla\overline{u_m}
       +u_n\overline{u_m}\big)\,dx
 =\delta_{nm}, \quad 1 \leq m,n \leq N.
$$
The one-body density of $\{u_n\}_{n=1}^N$ is defined in
\eqref{eq:one-body-density}. For $\alpha>0$, define 
\begin{equation}  \label{eq:Theta-def}
 \Theta_\alpha(s):=
 \begin{cases}
 \dfrac{e^{\alpha s}-1-\alpha s}{s},&s>0, \\[2mm]
 0,&s=0.
 \end{cases}
\end{equation}

Our first result provides an endpoint  Moser--Trudinger type inequality for the density of $H^1$-orthonormal families.
\begin{theorem}[Whole-space endpoint Moser--Trudinger type inequality]
\label{thm:intro-whole-space}
There is a universal constant $C>0$ such that for any
$H^1$-orthonormal family $\{u_n\}_{n=1}^N \subset H^1(\R^2)$ with
density $\rho$, one has
\begin{equation}  \label{eq:rank-one-defect}
 \int_{\R^2}\Theta_{4\pi}(\rho)\,dx
 \le C\left(N-1+\int_{\R^2}\rho\,dx\right).
\end{equation}
The endpoint coefficient $4\pi$ in \eqref{eq:rank-one-defect} is sharp
in the sense that for every $\alpha>4\pi$ and every
$N\ge1$,
\begin{equation}
 \sup_{\substack{\{u_n\}_{n=1}^N\subset H^1(\R^2)\\
                   H^1\text{-orthonormal}}}
 \int_{\R^2}\Theta_\alpha\!\left(\sum_{n=1}^N|u_n|^2\right)\,dx
 =\infty.
 \label{eq:intro-alpha-sharp}
\end{equation}
For $N\ge1$, put
\begin{equation}
 K_N=\sup\left\{
 \int_{\R^2}\Theta_{4\pi}\!\left(\sum_{n=1}^N|u_n|^2\right)\,dx\;\middle|\;
 \{u_n\}_{n=1}^N\text{ is $H^1$-orthonormal}\right\}.
 \label{eq:KN-definition}
\end{equation}
Then there exist two universal constants $c_1,c_2>0$ such that for any $N \geq 1$,
\begin{equation} \label{eq:KN-two-sided}
 c_1 N \leq K_N \leq c_2 N.
\end{equation}
\end{theorem}
The proof of Theorem~\ref{thm:intro-whole-space} is completed in
Subsection~\ref{subsec:endpoint}.

The right-hand side of~\eqref{eq:rank-one-defect}
is a mixed rank--mass
quantity.  The $H^1$-orthonormality gives
$\int_{\R^2}\rho\,dx \leq N$, which in turn leads to 
$N-1+\int_{\R^2}\rho\,dx \leq 2N-1$. Therefore,
\eqref{eq:rank-one-defect} implies a rank-only  bound
 \[ \int_{\R^2}\Theta_{4\pi}(\rho)\,dx
 \le C N,
\]
with another universal constant $C$.
Moreover, the term $N-1$ is needed in the sense that no universal constant can bound the left-hand side by the mass alone (see Proposition~\ref{prop:no-rank-uniform-mass}). 

Theorem~\ref{thm:intro-whole-space} exhibits two distinct sharpness statements. Inequality
\eqref{eq:intro-alpha-sharp} asserts that increasing $4\pi$ to any
$\alpha>4\pi$ makes the supremum infinite even when $N=1$, hence orthonormality cannot improve the endpoint coefficient $4\pi$.  At the
endpoint coefficient, the order of $N$ on the right-hand side of \eqref{eq:rank-one-defect} is sharp, as shown by
\eqref{eq:KN-two-sided}. As a consequence of
\eqref{eq:rank-one-defect}, one gets the subcritical estimate (see Corollary~\ref{cor:ruf-subcritical}) which states that there exists a universal $C>0$ such that for any $\varepsilon \in (0,1/2]$ and any $H^1$-orthonormal family $\{u_n\}_{n=1}^N \subset H^1(\R^2)$ with
density $\rho$, 
\begin{equation} \label{est:subcrt}
 \int_{\R^2}(e^{4\pi(1-\varepsilon)\rho}-1)\,dx
 \le \frac{C}{\varepsilon}
 N^{1-\frac{\varepsilon}{2}}
 \left(\int_{\R^2}\rho\,dx\right)^{\frac{\varepsilon}{2}}.
\end{equation}
The bound \eqref{est:subcrt} can be viewed as a whole-space counterpart of \eqref{eq:ann-subcritical} with $\varepsilon^{-1}$ dependence and mixed rank--mass factor
$N^{1-\varepsilon/2}\left(\int_{\R^2}\rho\,dx\right)^{\varepsilon/2}$.

When $N=1$, estimate~\eqref{eq:rank-one-defect} reduces to
\begin{equation} \label{eq:critical-N-one}
 \int_{\R^2}\Theta_{4\pi}(|u|^2)\,dx\le C\|u\|_{L^2}^2 \quad \text{for all } u \in H^1(\R^2) \text{ with }
\|u\|_{H^1}=1.
\end{equation}
It is worth pointing out that estimate \eqref{eq:critical-N-one} is equivalent to the Ibrahim--Masmoudi--Nakanishi inequality \eqref{eq:IMN-exact-growth-intro} and also implies Ruf's inequality \eqref{eq:ruf} (see
Corollary~\ref{cor:ruf-from-exact-growth}, together with the dilation argument in Remark~\ref{rem:scalar-comparisons}).

We further investigate inequality \eqref{eq:rank-one-defect} through a semiclassical analysis of the density generated by normalized Dirichlet eigenfunctions in different regimes.

Throughout the paper, $B_R(x)$ is the open ball of radius $R$ centered at $x$, and $B_R:=B_R(0)$. The Dirichlet Laplacian $-\Delta$ on $B_1$ has eigenvalues
$\mu_1\leq\mu_2\leq\cdots$
(counting multiplicities) and corresponding eigenfunctions
$\{\psi_j\}_{j\ge1}\subset H_0^1(B_1)$ forming an orthonormal basis of
$L^2(B_1)$.

For $R>0$, define 
 \begin{equation} \label{u_j^(R)}
 u_j^{(R)}(x):=
 \begin{cases}
 R^{-1}(1+\mu_j/R^2)^{-1/2}\psi_j(x/R),&x\in B_R,\\[1mm]
 0,&x\in\R^2\setminus B_R.
 \end{cases}
\end{equation}
They form an $H^1(\R^2)$-orthonormal family (see Lemma~\ref{lem:dirichlet-family}). Let \begin{equation} \label{rho_NR}
\rho_{N,R}:=\sum_{j=1}^N|u_j^{(R)}|^2
\end{equation} 
be the density of $\{u_j^{(R)}\}_{j=1}^N$. The terms $\int_{\R^2}\rho_{N,R}\, dx$ and $\int_{\R^2}\Theta_{4\pi}(\rho_{N,R})\, dx$ are referred to as the \textit{mass} and the \textit{energy}, respectively.

Let $\{\kappa_N\}_{N\ge1}$ be a sequence of positive numbers and set $R_N=\sqrt{\kappa_N N}$. Since $N/|B_{R_N}|=1/(\pi\kappa_N)$, $1/\kappa_N$ measures the number of states per unit area (up to the factor $1/\pi$). Therefore the asymptotic behavior of $\{\kappa_N\}_{N \geq 1}$ determines the density regime. More precisely, the \emph{intermediate-density regime} corresponds to
$\kappa_N \to \kappa \in (0,\infty)$. The \emph{low-density}, or \emph{dilute}, regime corresponds to $\kappa_N \to \infty$, whereas the \emph{high-density regime} corresponds to $\kappa_N \to 0$ and $R_N\to\infty$.

The next theorem gives the limit of the mass and endpoint energy per state
in the  intermediate-density regime.

 \begin{theorem}[Intermediate-density regime]
\label{thm:intro-intermediate-density}
Let $\{\kappa_N\}_{N\ge1}$ be a sequence of positive numbers such that
$\kappa_N \to \kappa \in (0,\infty)$ and set $R_N=\sqrt{\kappa_N N}$.
Let $\{u_j^{(R_N)}\}_{j=1}^N$ and $\rho_{N,R_N}$ be defined by
\eqref{u_j^(R)} and \eqref{rho_NR}, respectively. Then
 \begin{equation} \label{lim:Dirichlet}
 \lim_{N \to \infty}\frac{1}{N}\int_{\R^2}\rho_{N,R_N}\,dx
=  \frac{\kappa}{4}\ln\Big(1+\frac{4}{\kappa}\Big).
\end{equation}
Moreover,
 \begin{equation} \label{est:coupling}
\lim_{N \to \infty} \frac{1}N\int_{\R^2}
 \Theta_{4\pi}(\rho_{N,R_N}(x))\,dx
 =\frac{4\pi^2\bigl(4-\kappa\ln(1+4/\kappa)\bigr)}{\ln(1+4/\kappa)}.
\end{equation}
\end{theorem}

The proof of Theorem~\ref{thm:intro-intermediate-density} is completed in
Subsection~\ref{subsec:fixed-energy}. 

Interestingly, the family $\{u_j^{(R_N)}\}_{j=1}^N$  asymptotically attains the maximal mass per state localized in  $B_{R_N}$ over all $N$-element
$H^1$-orthonormal families, as stated precisely below.

For
$R>0$, define
\begin{equation} \label{eq:intro-DN-def}
 \mathfrak{m}_N(R):=\sup\left\{\int_{B_R}\sum_{n=1}^N|u_n(x)|^2\,dx:
 \{u_n\}_{n=1}^N\text{ is $H^1$-orthonormal in }H^1(\R^2)\right\}.
\end{equation}

\begin{theorem}[Asymptotic variational optimality in the intermediate-density regime]
\label{thm:intro-variational-optimality}
Let $\{\kappa_N\}_{N\ge1}$ be a sequence of positive numbers such that
$\kappa_N\to\kappa\in(0,\infty)$ and set $R_N=\sqrt{\kappa_N N}$. Then
\begin{equation}  \label{eq:variational-optimality}
 \lim_{N\to\infty}\frac{\mathfrak{m}_N(R_N)}N
=\frac{\kappa}{4}\ln\left(1+\frac{4}{\kappa}\right).
\end{equation}
\end{theorem}

The proof of Theorem \ref{thm:intro-variational-optimality} is given in Subsection~\ref{subsec:variational-optimality}.

In the low-density, or dilute, regime where $R_N=\sqrt{\kappa_N N}$ with $\kappa_N \to  \infty$, we compute explicit limits of the mass and energy per state. Moreover, we show that the normalized density $\kappa_N\rho_{N,R_N}$ converges to $1/\pi$ in the averaged $L^p$ norm for any $1 \leq p <\infty$. 

\begin{theorem}[Low-density regime]
\label{thm:low-density}
Let $\{\kappa_N\}_{N\ge1}$ be a sequence of positive numbers such that
$\kappa_N\to\infty$ and set $R_N=\sqrt{\kappa_N N}$.
Let $\{u_j^{(R_N)}\}_{j=1}^N$ and $\rho_{N,R_N}$ be defined by
\eqref{u_j^(R)} and \eqref{rho_NR}, respectively. Then there is a positive constant $C$
independent of $N$ and $\kappa_N$ such that
\begin{equation}\label{eq:dilute-global-bound}
 0\le\kappa_N\rho_{N,R_N}(x)\le C\qquad \text{for all } x \in \R^2.
\end{equation}
Moreover,
\begin{equation}\label{eq:dilute-mass}
 \lim_{N\to\infty}\frac{1}N\int_{\R^2}\rho_{N,R_N}\,dx=1.
\end{equation}
In addition, for any $1\le p<\infty$,
\begin{equation}\label{eq:dilute-Lp}
 \lim_{N\to\infty}
 \left(\frac{1}{|B_{R_N}|}\int_{B_{R_N}}
 \left|\kappa_N\rho_{N,R_N}(x)-\frac{1}{\pi}\right|^p\,dx\right)^{\frac{1}{p}}=0.
\end{equation}
Finally,
\begin{equation}\label{eq:dilute-energy}
 \lim_{N \to \infty}\frac{1}N\int_{\R^2}\Theta_{4\pi}(\rho_{N,R_N}(x))\,dx
 =8\pi^2.
\end{equation}
These conclusions are independent of the choice of the selected subspace within the eigenspace associated with  $\mu_N$.
\end{theorem}

 The mass limit~\eqref{eq:dilute-mass} has error $O(\kappa_N^{-1})$,
as shown in~\eqref{eq:dilute-mass-squeeze}. The  averaged convergence
in~\eqref{eq:dilute-Lp} holds for every finite $p\ge1$, but fails for
$p=\infty$ (see Remark~\ref{remark:dilute}). The proof of Theorem \ref{thm:low-density} is given in Section \ref{sec:low-density}.

 We next consider the high-density regime, where $\kappa_N \to 0$ and  $R_N=\sqrt{\kappa_N N} \to \infty$. Interestingly, in this case, both the mass and the endpoint energy per state converge to zero 
  \begin{align*}
   \lim_{N\to\infty}\frac1N\int_{\R^2}\rho_{N,R_N}(x)\,dx
   &=0,\\
   \lim_{N\to\infty}\frac1N\int_{\R^2}
   \Theta_{4\pi}(\rho_{N,R_N}(x))\,dx
   &=0.
  \end{align*}
  These limits are proved in Proposition~\ref{prop:converge-zero} and differ from those in the intermediate-density and low-density regimes. In the high-density regime, the suitable normalized density is $\rho_{N,R_N}/\ln(1+4/\kappa_N)$.

\begin{theorem}[High-density regime]
\label{thm:intro-high-density}
Let $\{\kappa_N\}_{N\ge1}$ be a sequence of positive numbers such that
$\kappa_N\to0$. Set $R_N=\sqrt{\kappa_N N}$ and assume that
$R_N\to\infty$. Let $\{u_j^{(R_N)}\}_{j=1}^N$ and $\rho_{N,R_N}$ be defined by
\eqref{u_j^(R)} and \eqref{rho_NR}, respectively. Put $\beta_N=\ln(1+4/\kappa_N)$. Then, for
every $1\le p<\infty$,
\begin{equation} \label{eq:Lp-highdensity}
 \lim_{N \to \infty}\left(\frac{1}{|B_{R_N}|}\int_{B_{R_N}}
 \left|\frac{\rho_{N,R_N}}{\beta_N}-\frac{1}{4\pi}\right|^p\,dx\right)^{\frac{1}{p}}
= 0.
\end{equation}
Moreover,
\begin{equation} \label{eq:sup-rho-limit}
\lim_{N \to \infty} \left( \frac{1}{\beta_N}\sup_{x\in\R^2}\rho_{N,R_N}(x) \right) = \frac{1}{4\pi},
\end{equation}
and
\begin{equation}  \label{eq:intro-Theta-convergence}
 \lim_{N\to\infty}\frac{1}{|B_{R_N}|}\int_{B_{R_N}}
 \Theta_{4\pi}\left(\frac{\rho_{N,R_N}(x)}{\beta_N}\right)\,dx
 =4\pi(e-2).
\end{equation}
\end{theorem}

 Equation~\eqref{eq:Lp-highdensity} gives convergence in every finite averaged
$L^p$  norm, while~\eqref{eq:sup-rho-limit} identifies the limiting normalized
supremum. The proof of
Theorem~\ref{thm:intro-high-density} is completed  in
Section~\ref{sec:high-density}.

\subsection{Difficulties, proof strategy and novelties}
We explain the main difficulties and the ideas of the proofs.
\smallskip

\noindent  $\bullet$ \textit{The endpoint inequality and sharpness.} In the proof of ~\eqref{eq:rank-one-defect}, we face several difficulties. The first comes from the interaction of the functions in orthonormal families, which significantly complicates the analysis. The second one stems from the endpoint coefficient $4\pi$, which requires delicate estimates. The third one arises in the whole-space setting, where tools such as Poincar\'e's inequality (an important ingredient in \cite{ANN}) are unavailable. The last difficulty originates from the mixed rank-mass form of the bound on the right-hand side of ~\eqref{eq:rank-one-defect}. 

To overcome these challenges, we combine two independent inequalities:  a rank-only Moser--Trudinger type inequality for the density $\rho$ of orthonormal families $\{u_n\}_{n=1}^N$ (see estimate \eqref{eq:ruf-endpoint}) and a scalar Moser--Trudinger type inequality with a mass-dependence bound (see estimate~\eqref{eq:scalar-exact-growth}). In order to obtain the first inequality, we develop a heat-semigroup refinement of Rumin's density method using the heat-regularized density $\rho_t(x)=\sum_{n=1}^N\big|(e^{-t\cA}u_n)(x)\big|^2$ where $\cA=-\Delta+1$ on $\R^2$.  A novel aspect of our approach is that it does not rely on the Fourier transform which is a key tool in the classical Rumin method. More precisely, by combining the layer-cake identity~\eqref{eq:semigroup-layer-cake}
with the reverse triangle inequality, we obtain~\eqref{eq:rumin-semigroup},
which bounds the integral of $(\sqrt{\rho}-\sqrt{\rho_t})^2$
over $\R^2\times(0,\infty)$ by a constant
multiple of $N$.
Thanks to Bessel's inequality~\eqref{eq:bessel-Q} and
the logarithmic kernel estimate~\eqref{eq:Qt-log-final}, we then
bound $\rho_t$ and derive the rank-only endpoint
estimate~\eqref{eq:ruf-endpoint} through Proposition~\ref{prop:abstract-endpoint}.
The coefficient of the logarithmic bound determines the sharp
exponential coefficient $4\pi$. Then we combine a rearrangement argument,
the quantitative exterior estimate~\eqref{eq:exterior-descent}, and
the classical Moser--Trudinger inequality on a ball, to establish the scalar Moser--Trudinger  estimate~\eqref{eq:scalar-exact-growth}. 

By using a concentration argument, we establish the divergence
in~\eqref{eq:intro-alpha-sharp}, thereby confirming the optimality of the coefficient $4\pi$. Moreover, employing a technique based on disjoint translations, we derive the lower bound needed for the sharp linear growth~\eqref{eq:KN-two-sided}. It is worth emphasizing that in the case $N=1$, our inequality ~\eqref{eq:rank-one-defect} is equivalent to the Ibrahim--Masmoudi--Nakanishi inequality \eqref{eq:IMN-exact-growth-intro} and also implies Ruf's inequality \eqref{eq:ruf} (see
Corollary~\ref{cor:ruf-from-exact-growth} and Remark \ref{rem:scalar-comparisons}).
To the best of our knowledge, estimate ~\eqref{eq:rank-one-defect} is the first whole-space endpoint Moser--Trudinger type inequality for orthonormal families.
\smallskip

\noindent $\bullet$ \textit{Intermediate-density regime.}
In this regime, we compute the explicit limits ~\eqref{lim:Dirichlet} and \eqref{est:coupling}. The mass-per-state limit~\eqref{lim:Dirichlet} is proved by combining Weyl's law with the mass identity~\eqref{eq:rhoN-dirichlet}.

The derivation of the energy-per-state convergence \eqref{est:coupling} requires more delicate analysis to treat the integrals on the interior ball and near the boundary separately. A difficulty we encounter is that truncating at the $N$th eigenfunction may involve only part of an eigenspace. By separating this partial
cluster (see Lemma~\ref{lem:sublinear-cluster}) and applying  H\"ormander's local spectral estimate, we establish the
uniform interior limit~\eqref{eq:bulk-limit} on $B_{(1-\delta)R_N}$ for every fixed $\delta \in (0,1)$, independently of the selected subspace.
On the other hand, we establish a global density bound (see Lemma~\ref{lem:uniform-rhoN}) to control the
remaining boundary layer. Combining these two estimates allows us to derive  the explicit energy limit~\eqref{est:coupling} for every admissible truncation. It is noteworthy that this result does not
 require $N$ to coincide with the end of a spectral cluster.

As a consequence of ~\eqref{lim:Dirichlet} and the convexity of $\Theta_{4\pi}$, we obtain Proposition~\ref{prop:no-rank-uniform-mass}, which 
clarifies the role of the rank term $N-1$ on the right-hand side of  
estimate~\eqref{eq:rank-one-defect}.
\smallskip

\noindent $\bullet$ \textit{Asymptotic variational optimality in the intermediate-density regime.} A notable feature of the family  $\{u_j^{(R_N)}\}_{j=1}^N$ is that it asymptotically attains the maximal mass per state localized in  $B_{R_N}$ over all $N$-element
$H^1$-orthonormal families, as asserted in~\eqref{eq:variational-optimality}. To show this, we rely on  Ky Fan's principle~\cite{Fan49} to reduce the supremum to
the sum of the largest eigenvalues of the restricted resolvent,
as expressed in~\eqref{KyFan}.
A difficulty in the proof is that an arbitrary admissible whole-space $H^1$-orthonormal family in the definition \eqref{eq:intro-DN-def} need not vanish
on the boundary of the ball. By using an argument based on the restriction of variational problems, we obtain the lower
bound~\eqref{eq:variational-lower} for $\mathfrak{m}_N(R)$, while from IMS localization on a slightly larger ball, we infer the upper bound~\eqref{eq:variational-ims-sum}.
We choose the enlarged radius so that its relative change and the
localization error both tend to zero. Therefore the annulus $B_{R'} \setminus B_R$ (see Lemma \ref{lem:variational-ims-upper}) must have
vanishing relative width to preserve the limiting parameter $\kappa$, while
its absolute width must grow to make the IMS error negligible.
By applying the mass
limit~\eqref{eq:density-limit} to both radii, we obtain the same
asymptotic value. This proves optimality
without restricting the $H^1$-orthonormal families to Dirichlet eigenfunctions.
\smallskip

\noindent $\bullet$ \textit{The low-density regime.} In the proof of the averaged $L^p$ convergence~\eqref{eq:dilute-Lp}, a difficulty is  decoupling the spatial density profile~\eqref{q_N} from eigenvalue-dependent normalization weights~\eqref{eq:lambdaj}. Our strategy isolates the unweighted projector's spectral error~\eqref{est:S/N-pi} from these weight discrepancies~\eqref{q_N:bulk}. By combining a local Weyl law bound for the unweighted projector~\eqref{eq:dilute-projector-bulk} with the error bound of size $O(\kappa_N^{-1})$, we obtain a uniform interior limit~\eqref{eq:dilute-bulk} without imposing any rate restriction on $\kappa_N$ relative to $N$. Then we  apply a global upper bound~\eqref{eq:dilute-global-bound} to ensure that the boundary contribution vanishes in finite averaged norms. Together, the interior limit and the vanishing boundary contribution lead to the convergence~\eqref{eq:dilute-Lp}. The remaining limits in Theorem~\ref{thm:low-density} uses distinct approaches. We prove the mass limit~\eqref{eq:dilute-mass} through a squeeze  argument~\eqref{eq:dilute-mass-squeeze} on the normalization weights, bypassing spatial analysis entirely. Finally, because the global bound forces the density uniformly  to zero, by using a Taylor expansion, we reduce the energy limit to its leading-order mass term, thereby deriving~\eqref{eq:dilute-energy}. \smallskip

\noindent $\bullet$ \textit{The high-density regime.}
In this regime where $\kappa_N\to0$ and $R_N=\sqrt{\kappa_N N}\to\infty$, 
both the mass and endpoint energy per state vanish
(see Proposition~\ref{prop:converge-zero}). To describe the spatial
density, we set $\beta_N=\ln(1+4/\kappa_N)$ and study the normalized density
$\rho_{N,R_N}/\beta_N$, which is suggested by
$\rho_\infty(\kappa_N)=\beta_N/(4\pi)$
(see~\eqref{eq:rho-infty-def}). With this suitable normalization, we establish two key results: the global upper bound~\eqref{eq:uniform-upper-claim} ensures that any positive error of $\rho_{N,R_N}/\beta_N$ from $1/(4\pi)$ is uniformly small, while the averaged mass limit~\eqref{eq:normalized-mass-highdensity} forces the error to vanish. Together, these two facts establish the averaged $L^1$ convergence~\eqref{eq:Lp-highdensity}. The uniform upper bound then extends this convergence to every finite $p\ge1$. Furthermore, because the spatial average provides a natural lower bound for the supremum, combining this averaged mass limit with the global upper bound \eqref{eq:uniform-upper-claim} yields~\eqref{eq:sup-rho-limit}. Finally, since $\Theta_{4\pi}$ is Lipschitz on the common bounded range of the normalized density, from the averaged $L^1$ convergence, we derive~\eqref{eq:intro-Theta-convergence}.

\medskip
 \noindent\textbf{Organization.}
In Section~\ref{sec:semigroup-bessel}, we develop the abstract semigroup framework. Section~\ref{sec:endpoint} establishes the endpoint and subcritical inequalities and the sharpness properties. In Section~\ref{sec:dirichlet-families}, we construct the Dirichlet families and record the spectral estimates. Section~\ref{sec:intermediate-density} proves the intermediate-density mass and energy limits and the obstruction to mass-only bounds, with asymptotic variational optimality established in Subsection~\ref{subsec:variational-optimality}. Sections~\ref{sec:low-density} and~\ref{sec:high-density} treat the low- and high-density regimes. Appendix~\ref{app:scalar-exact-growth} derives the scalar endpoint inequality from the classical Moser--Trudinger theorem on a ball and proves the scalar comparison results. Appendix~\ref{sec:auxiliary-estimates} collects the auxiliary semigroup estimates.

\medskip
 \noindent\textbf{Acknowledgement.} Tuan-Dat Tran  acknowledges the support of the Specific Research project SV No. MUNI/A/1648/2025.

\section{Abstract framework}
\label{sec:semigroup-bessel}

We combine a semigroup identity for nonnegative self-adjoint operators
with the reverse triangle inequality in $\ell^2$ to convert
orthonormality into an integral estimate for the density $\rho$ and its
heat-flow regularizations $\rho_t$. Then we use Bessel's inequality to control the
regularized density by a diagonal kernel, thereby obtaining
Proposition~\ref{prop:abstract-endpoint} under abstract hypotheses.
These abstract results will be applied to $\cA=-\Delta+1$ on
$L^2(\R^2)$ in Subsection~\ref{subsec:diagonal-bound} and to its Dirichlet realizations in Section~\ref{sec:intermediate-density}.

\begin{lemma}[Semigroup layer-cake identity]
\label{lem:layer-cake}
Let $\cA$ be a nonnegative self-adjoint operator on a Hilbert space $\scrH$, and set $\scrQ=\operatorname{Dom}(\cA^{1/2})$. For every $u\in\scrQ$,
\begin{equation}
\|\cA^{1/2}u\|_{\scrH}^2
=
\frac{1}{2\ln 2}\int_0^\infty\big\|(1-e^{-t\cA})u\big\|_{\scrH}^2\,\frac{dt}{t^2}.
\label{eq:semigroup-layer-cake}
\end{equation}
\end{lemma}

\begin{proof}
Let $E_\cA$ be the spectral measure of $\cA$. By the spectral theorem,
$\|(1-e^{-t\cA})u\|_{\scrH}^2=\int_{(0,\infty)}(1-e^{-t\lambda})^2\,d\langle u,E_\cA(\lambda)u\rangle$. Using Tonelli's theorem and Lemma~\ref{lem:scalar-identity}, we obtain
\begin{align*}
\int_0^\infty\big\|(1-e^{-t\cA})u\big\|_{\scrH}^2\,\frac{dt}{t^2}
& =
\int_{(0,\infty)}\left[\int_0^\infty\frac{(1-e^{-t\lambda})^2}{t^2}\,dt\right]
d\langle u,E_\cA(\lambda)u\rangle \\
& = 2(\ln 2)\int_{(0,\infty)}\lambda\,d\langle u,E_\cA(\lambda)u\rangle
=2(\ln 2)\|\cA^{1/2}u\|_{\scrH}^2,
\end{align*}
which implies \eqref{eq:semigroup-layer-cake}. The proof is complete.
\end{proof}

Let $(X,dx)$ be a $\sigma$-finite measure space and let $\cA$ be a  positive self-adjoint operator on $L^2(X)$  such that $\cA \geq a_* I$ for some $a_*>0$. Set $\scrQ:=\operatorname{Dom}(\cA^{1/2})$. A family
$\{u_n\}_{n=1}^N \subset \scrQ$ is called \emph{$\scrQ$-orthonormal} if
\[
\langle \cA^{1/2}u_n,\cA^{1/2}u_m\rangle_{L^2(X)}=\delta_{nm}, \quad   1 \leq n,m \leq N.
\]
We assume that, for every $u\in\scrQ$, the $L^2(X)$-valued heat mapping $t\mapsto e^{-t\cA}u$ admits a jointly measurable representative $(t,x)\mapsto(e^{-t\cA}u)(x)$ on $(0,\infty)\times X$.
Define
\begin{equation} \label{abstract:rho}
\rho(x):=\sum_{n=1}^N|u_n(x)|^2,
\quad
\rho_t(x):=\sum_{n=1}^N\big|(e^{-t\cA}u_n)(x)\big|^2,
\quad t>0.
\end{equation}

\begin{proposition}[Semigroup lower bound]
\label{prop:rumin}
For every $\scrQ$-orthonormal family $\{u_n\}_{n=1}^N$, we have
\begin{equation}
N
 \geq
\frac{1}{2\ln 2}\int_X\int_0^\infty
\big(\sqrt{\rho(x)}-\sqrt{\rho_t(x)}\big)^2\,\frac{dt}{t^2}\,dx.
\label{eq:rumin-semigroup}
\end{equation}
\end{proposition}
\begin{proof}
Let $\{u_n\}_{n=1}^N$ be a $\scrQ$-orthonormal family. For  almost every $(x,t)\in X\times(0,\infty)$, the reverse triangle inequality  in $\ell^2$ gives
\begin{equation} \label{eq:reverse-triangle-pointwise}
\sum_{n=1}^N\big|u_n(x)-e^{-t\cA}u_n(x)\big|^2
\ \ge\
\big(\sqrt{\rho(x)}-\sqrt{\rho_t(x)}\big)^2.
\end{equation}
Summing~\eqref{eq:semigroup-layer-cake} (applied to each $u_n$) over
$n=1,\dots,N$, then using $\|\cA^{1/2}u_n\|_{L^2(X)}=1$ and \eqref{eq:reverse-triangle-pointwise}, we have
\begin{align*}
N & =\sum_{n=1}^N\|\cA^{1/2}u_n\|_{L^2(X)}^2
= \frac{1}{2\ln 2}\int_0^\infty\sum_{n=1}^N
\big\|(1-e^{-t\cA})u_n\big\|_{L^2(X)}^2\,\frac{dt}{t^2} \\
&= \frac{1}{2\ln 2}\int_0^\infty
\int_X\sum_{n=1}^N|u_n(x)-e^{-t\cA}u_n(x)|^2\,dx \,\frac{dt}{t^2} \geq \frac{1}{2\ln 2}\int_X\int_0^\infty
\big(\sqrt{\rho(x)}-\sqrt{\rho_t(x)}\big)^2\,\frac{dt}{t^2}\,dx,
\end{align*}
which is~\eqref{eq:rumin-semigroup}. The proof is complete.
\end{proof}

The next result provides a pointwise upper bound for $\rho_t$.

\begin{lemma}
\label{lem:bessel}
Let $\cA \ge a_* I$ (for some $a_*>0$) be self-adjoint on $L^2(X)$ and suppose that there is a jointly measurable complex-valued function
$G$ on $(0,\infty)\times X\times X$ such that, for every $t>0$, $G_t(x,y)=G(t,x,y)$ is a kernel of $\cA^{-1/2}e^{-t\cA}$ and
$G_t(x,\cdot)\in L^2(X)$ for a.e. $x \in X$. Then, for every
$\scrQ$-orthonormal family $\{u_n\}_{n=1}^N$, every $t>0$ and a.e. $x \in X$,
\begin{equation} \label{eq:bessel-Q}
\rho_t(x)\ \le\ Q_t(x),
\quad \text{where }
Q_t(x)=\|G_t(x,\cdot)\|_{L^2(X)}^2.
\end{equation}
\end{lemma}
\begin{proof}
Set $v_n:=\cA^{1/2}u_n$. By the $\scrQ$-orthonormality of $\{u_n\}_{n=1}^N$,
$\langle v_n,v_m\rangle_{L^2(X)}=\delta_{nm}$ for all $1 \leq m,n \leq N$, hence $\{v_n\}_{n=1}^N$ is $L^2(X)$-orthonormal. Since $\cA^{1/2}$ commutes with $e^{-t\cA}$, it follows that
$e^{-t\cA}u_n=e^{-t\cA}\cA^{-1/2}v_n=\cA^{-1/2}e^{-t\cA}v_n$.

Now fix $t>0$. For a.e. $x \in X$, by the assumption that $\cA^{-1/2}e^{-t\cA}$ has kernel
$G_t(x,\cdot)\in L^2(X)$, we have
\[
(\cA^{-1/2}e^{-t\cA}v_n)(x)
=\int_X G_t(x,y)v_n(y)\,dy
=\langle v_n,\overline{G_t(x,\cdot)}\rangle_{L^2(X)}.
\]
Thanks to Bessel's inequality for the
orthonormal system $\{v_n\}_{n=1}^N$, we derive
 \[
\rho_t(x)=\sum_{n=1}^N\big|(e^{-t\cA}u_n)(x)\big|^2
=
\sum_{n=1}^N\big|\langle v_n,\overline{G_t(x,\cdot)}\rangle\big|^2
\leq
\big\|G_t(x,\cdot)\big\|_{L^2(X)}^2 = Q_t(x).
\]
The proof is complete.
\end{proof}

\begin{proposition}
\label{prop:abstract-endpoint}
Under the assumptions of Lemma~\ref{lem:bessel}, suppose additionally that for some $a,\tau>0$,
\begin{equation} 
Q_t(x)\ \le\ a\ln\Big(1+\frac{\tau}{t}\Big)
\quad \text{for a.e. }(t,x)\in(0,\infty)\times X.
\label{eq:log-Q-hyp}
\end{equation}
Let $\{u_n\}_{n=1}^N$ be a $\scrQ$-orthonormal family
with density $\rho=\sum_n|u_n|^2$.  Then
\begin{equation}
 \int_X\Theta_{1/a}(\rho(x))\,dx
 \le
 \left(\frac{e}{2a^2a_*}+\frac{2C_0\tau}{a^2}\right)N,
\label{eq:abstract-endpoint}
\end{equation}
where $C_0>0$ is a universal constant, independent of $a,\tau,N,X$ and of
the operator $\cA$.
\end{proposition}

\begin{proof}
 Since $\rho_t(x)$ and $Q_t(x)$ are jointly measurable, by the Bessel inequality, \eqref{eq:log-Q-hyp} and Fubini's theorem, we derive $\rho_t(x)\le Q_t(x)\le a\ln(1+\tau/t)$ for a.e. $(t,x)$. Hence, for a.e. $x\in X$, this estimate holds for  almost every $t>0$.
Fix such an $x$ and put  $\varrho:=\rho(x)/a$. On the set of
$t>0$ where  $\ln(1+\tau/t)<\varrho$ (equivalently,
$a\ln(1+\tau/t)<\rho(x)$), we have $\rho_t(x)<\rho(x)$. Moreover, since
$\rho_t(x)\le a\ln(1+\tau/t)$,
 \[
\sqrt{\rho(x)}-\sqrt{\rho_t(x)}
\ \ge\
\sqrt{\rho(x)}-\sqrt{a\ln(1+\tau/t)}
=
\sqrt a\Big(\sqrt \varrho-\sqrt{\ln(1+\tau/t)}\Big).
\]
Hence, by the change of variable $t=\tau/(e^s-1)$, we derive
 \begin{align*}
\int_0^\infty (\sqrt{\rho(x)}-\sqrt{\rho_t(x)})^2 \frac{dt}{t^2} &\geq \int_{\{ t>0: \, \ln(1+\tau/t) < \varrho \}} (\sqrt{\rho(x)}-\sqrt{\rho_t(x)})^2 \frac{dt}{t^2}  \\
&\geq \int_0^\varrho\Big[\sqrt a\big(\sqrt \varrho-\sqrt s\big)\Big]^2\frac{e^s}{\tau}\,ds
\geq
c \frac{a}{\tau}\Big(\frac{e^\varrho}{1+\varrho}-1\Big),
\end{align*}
where we have used Lemma~\ref{lem:J-lower} for the last inequality.
Integrating over $x\in X$ and applying~\eqref{eq:rumin-semigroup}, while
keeping in mind that  $\varrho=\rho(x)/a$, we obtain, with  $C_0:=2\ln2/c$,
\begin{equation} \label{est:exp-1}
 \int_X\left[\frac{\exp(\rho(x)/a)}{1+\rho(x)/a}-1\right]\,dx
 \le C_0\frac{\tau}{a}N.
\end{equation}

For $r>0$, put  $\sigma=r/a$ and
$$
 \mathcal B(\sigma):=\frac{e^\sigma}{1+\sigma}-1=\frac{e^\sigma-1-\sigma}{1+\sigma}.
$$
 By the definition of $\Theta_{1/a}$,
\begin{equation}
\Theta_{1/a}(r)=\frac{1}a\frac{1+\sigma}{\sigma}\mathcal B(\sigma).
\label{eq:Theta-semigroup-identity}
\end{equation}
At $r=0$, both $\Theta_{1/a}(r)$ and  $\mathcal B(r/a)$ vanish, and the comparison below follows by continuity.
If  $0\le \sigma\le1$, by Taylor's formula with integral remainder, we have
 $e^\sigma-1-\sigma\le e\sigma^2/2$, and hence
$\Theta_{1/a}(r)\le er/(2a^2)$.  If  $\sigma\ge1$, then
$(1+\sigma)/\sigma\le2$, so~\eqref{eq:Theta-semigroup-identity} gives
 $\Theta_{1/a}(r)\le2\mathcal B(\sigma)/a$. Consequently, for every $r\ge0$,
 \begin{equation}  \label{eq:Theta-semigroup-comparison}
 \Theta_{1/a}(r)
 \le\frac{e}{2a^2}r+\frac{2}a\mathcal B\Big(\frac{r}{a}\Big).
\end{equation}
Finally, thanks to the assumption $\cA\ge a_*I$ and $\scrQ$-orthonormality of $\{u_n\}_{n=1}^N$, we obtain
\begin{equation} \label{eq:mass-bound-abstract}
 \int_X\rho\,dx=\sum_{n=1}^N\|u_n\|_{L^2(X)}^2
 \leq \frac{1}{a_*}\sum_{n=1}^N\|\cA^{1/2}u_n\|_{L^2(X)}^2=\frac{N}{a_*}.
\end{equation}
 Integrate~\eqref{eq:Theta-semigroup-comparison} and  apply~\eqref{est:exp-1} and~\eqref{eq:mass-bound-abstract} to obtain~\eqref{eq:abstract-endpoint}.
\end{proof}

\section{Whole-space endpoint  inequality}
\label{sec:endpoint}

 We first apply Proposition~\ref{prop:abstract-endpoint}  to obtain a rank-only Moser--Trudinger type inequality. Combining it with  the scalar endpoint estimate,  whose proof is given in Appendix~\ref{app:scalar-exact-growth}, leads to the rank--mass bound \eqref{eq:rank-one-defect}. Then we prove the two sharpness assertions \eqref{eq:intro-alpha-sharp} and \eqref{eq:KN-two-sided}. Finally, we establish a subcritical rank--mass inequality.

 \subsection{Rank-only endpoint  bound}  \label{subsec:diagonal-bound}
Now consider $\cA=-\Delta+1$ on $L^2(\R^2)$ and let $Q_t$ be defined in \eqref{eq:bessel-Q}. Define the exponential integral
 \[
\Ei(\sigma):=\int_\sigma^\infty e^{-s}\frac{ds}{s}, \quad \sigma>0.
\]

\begin{lemma}
\label{lem:Qt-explicit}
Let $\cA=-\Delta+1$ on $L^2(\R^2)$. For every $t>0$ and $x\in\R^2$,
\begin{equation} \label{eq:Qt-explicit}
Q_t(x)\ =\ \frac{1}{4\pi}\,\Ei(2t).
\end{equation}
\end{lemma}
\begin{proof}
The heat kernel of $\cA=-\Delta+1$ is
$k_{\cA}(s,x,y)=e^{-s}(4\pi s)^{-1}e^{-|x-y|^2/(4s)}$.
In particular, $k_{\cA}(s,x,x)=\frac{e^{-s}}{4\pi s}$ for $x \in \R^2$. By using the identity
$e^{-2t\mu}/\mu=\int_{2t}^\infty e^{-s\mu}\,ds$ for $\mu>0$, we obtain
\[
Q_t(x)=(\cA^{-1}e^{-2t\cA})(x,x)=\int_{2t}^\infty k_{\cA}(s,x,x)\,ds
=
\int_{2t}^\infty\frac{e^{-s}}{4\pi s}\,ds
=
\frac{1}{4\pi}\Ei(2t),
\]
which implies~\eqref{eq:Qt-explicit}.
\end{proof}

\begin{lemma}
\label{lem:E1-bound}
For every $\sigma>0$, one has
$
\Ei(\sigma)\ \le \ln (1+\frac{1}{\sigma}).
$
\end{lemma}
\begin{proof}
By the elementary inequality $e^{-s}\le(1+s)^{-1}$ for $s\ge0$, we have
$e^{-s}/s\le1/(s(1+s))$ for $s>0$. This implies
\[
 \Ei(\sigma)\le \int_\sigma^\infty\frac{ds}{s(s+1)}
 =\ln\Big(1+\frac{1}{\sigma}\Big).
\]
The proof is complete.
\end{proof}

\begin{corollary} \label{cor:Qt-log}
For every $t>0$, one has
\begin{equation} \label{eq:Qt-log-final}
Q_t(x)\ \le\ \frac{1}{4\pi}\ln\Big(1+\frac{1}{2t}\Big), \quad \forall x \in \R^2.
\end{equation}
\end{corollary}
\begin{proof}
Estimate \eqref{eq:Qt-log-final} follows from Lemma~\ref{lem:Qt-explicit} and
Lemma~\ref{lem:E1-bound} with $\sigma=2t$.
\end{proof}
 The abstract endpoint Proposition~\ref{prop:abstract-endpoint}  gives the following rank-only bound.

 \begin{lemma}[Rank-only endpoint Moser--Trudinger type inequality]
\label{lem:ruf-endpoint}
For every $N\ge1$ and every $H^1$-orthonormal family $\{u_n\}_{n=1}^N$, we have
\begin{equation} \label{eq:ruf-endpoint}
 \int_{\R^2}\Theta_{4\pi}(\rho(x))\,dx\le CN,
\end{equation}
with $\Theta_{4\pi}$ as in~\eqref{eq:Theta-def} and a universal
constant $C>0$.
\end{lemma}
\begin{proof}
Corollary~\ref{cor:Qt-log} verifies~\eqref{eq:log-Q-hyp} with
$a=1/(4\pi)$ and $\tau=1/2$. Therefore, by using Proposition~\ref{prop:abstract-endpoint} with $a_*=1$, we obtain \eqref{eq:ruf-endpoint}.
\end{proof}

 \subsection{Scalar endpoint inequality and mixed rank--mass bound}
The scalar endpoint inequality provides the mass factor at rank one.

\begin{proposition}[Scalar endpoint Moser--Trudinger type inequality]
\label{prop:scalar-exact-growth}
There is a universal constant $C_{\rm sc}>0$ such that
\begin{equation} \label{eq:scalar-exact-growth}
 \int_{\R^2}\Theta_{4\pi}(|u|^2)\,dx
 \le C_{\rm sc}\|u\|_{L^2}^2 \quad \text{for all } u \in H^1(\R^2) \text{ with } \|\nabla u\|_{L^2} \leq 1.
\end{equation}
\end{proposition}

The proof of Proposition~\ref{prop:scalar-exact-growth} is deferred to Appendix~\ref{app:scalar-exact-growth}, in which we establish the
equivalence of~\eqref{eq:scalar-exact-growth}
and~\eqref{eq:IMN-exact-growth-intro} and derive Ruf's
inequality~\eqref{eq:ruf}.

The scalar estimate and the rank-only bound yield inequality~\eqref{eq:rank-one-defect}.
\begin{lemma} \label{lem:endpoint}
There is a universal constant $C>0$ such that for any
$H^1$-orthonormal family $\{u_n\}_{n=1}^N \subset H^1(\R^2)$ with density $\rho$, estimate~\eqref{eq:rank-one-defect} holds.
\end{lemma}
\begin{proof}
If $N=1$, put $u=u_1$. Then $\|u\|_{H^1}=1$ and
$\|\nabla u\|_{L^2}\le1$. With $\rho=|u|^2$, we infer from Proposition~\ref{prop:scalar-exact-growth}
that
\[
 \int_{\R^2}\Theta_{4\pi}(\rho)\,dx
 \le C_{\rm sc}\|u\|_{L^2}^2
 =C_{\rm sc}\int_{\R^2}\rho\,dx.
\]
If $N\ge2$, thanks to Lemma~\ref{lem:ruf-endpoint} and
$N\le2(N-1)$, we have
\[
 \int_{\R^2}\Theta_{4\pi}(\rho)\,dx
 \leq \tilde C N\le2 \tilde C(N-1).
\]
Taking $C:=\max\{C_{\rm sc},2\tilde C\}$, we derive 
\eqref{eq:rank-one-defect} in both cases.
\end{proof}

\subsection{Optimality of \texorpdfstring{$4\pi$}{4pi} and sharpness of linear growth in the rank} \label{subsec:endpoint}
The next lemma shows that the endpoint coefficient $4\pi$ in
Lemma~\ref{lem:ruf-endpoint} cannot be replaced by any larger
coefficient.
\begin{lemma}[Sharpness of the endpoint coefficient $4\pi$]
\label{lem:sharp-critical-coefficient}
For every $\alpha>4\pi$ and every $N\ge1$,
\begin{equation}
 \sup_{\substack{\{u_n\}_{n=1}^N\subset H^1(\R^2)\\
                   H^1\text{-orthonormal}}}
 \int_{\R^2}\Theta_\alpha\!\left(\sum_{n=1}^N|u_n(x)|^2\right)\,dx
 =\infty.
 \label{eq:sharp-critical-coefficient}
\end{equation}
\end{lemma}

\begin{proof}
It is enough to construct one concentrating state. For an integer $k\ge2$, define the radial concentrating function
\[
 w_k(x):=\frac{1}{\sqrt{2\pi}}
 \begin{cases}
  \sqrt{\ln k},&0\le |x|\le k^{-1},\\[2mm]
  \dfrac{\ln(1/|x|)}{\sqrt{\ln k}},&k^{-1}<|x|<1,\\[3mm]
  0,&|x|\ge1.
 \end{cases}
\]
Then $w_k$ is continuous, piecewise smooth, and belongs to
$H_0^1(B_1)$. By a direct calculation in polar coordinates, we can show that  $\|\nabla w_k\|_{L^2}^2=1$.
Moreover,
\[
 m_k:=\|w_k\|_{L^2}^2 =\frac{\ln k}{2k^2}
   +\frac{1}{\ln k}\int_{1/k}^1\Big(\ln\Big(\frac{1}{r}\Big)\Big)^2r\,dr \leq \frac{\ln k}{2k^2}+\frac{1}{4 \ln k}
 \to 0 \quad \text{as } k \to \infty.
\]
Hence $u_k:=\frac{w_k}{\sqrt{1+m_k}}$
satisfies $\|u_k\|_{H^1}=1$. On the ball $B_{1/k}$, $|u_k|^2=s_k:=\frac{\ln k}{2\pi(1+m_k)}$.
Set $q:=\alpha/(2\pi)>2$ and choose
$\eta\in(0,q-2)$. Since $m_k\to0$, for sufficiently large $k$, we have $\frac{q}{1+m_k}\ge2+\eta$.
Since $k^{q/(1+m_k)}$ dominates $1+\alpha s_k=O(\ln k)$, for sufficiently large $k$,
$k^{q/(1+m_k)}-1-\alpha s_k\ge\frac{1}{2}k^{q/(1+m_k)}$.
It follows that
\[
 \int_{\R^2}\Theta_\alpha(|u_k|^2)\,dx
 \ge \frac{\pi}{k^2}\Theta_\alpha(s_k)
 \ge \frac{\pi}{2k^2s_k}k^{\frac{q}{1+m_k}}
 =\frac{\pi^2(1+m_k)}{\ln k}
   k^{\frac{q}{1+m_k} - 2}
 \ge \frac{\pi^2k^\eta}{\ln k}
 \to \infty.
\]
This proves~\eqref{eq:sharp-critical-coefficient} when $N=1$.

For a fixed $N\ge2$, choose $N-1$ smooth, compactly supported
$H^1$-normalized functions $\varphi_2,\ldots,\varphi_N$ whose supports
are mutually disjoint and lie outside $\overline{B_1}$. Then
$\{u_k,\varphi_2,\ldots,\varphi_N\}$ is $H^1$-orthonormal. From the
power series of $\Theta_\alpha$, we find that 
$\Theta_\alpha$ is increasing. Therefore
\[
 \int_{\R^2}\Theta_\alpha\!\left(
 |u_k|^2+\sum_{n=2}^N|\varphi_n|^2\right)\,dx
 \ge \int_{\R^2}\Theta_\alpha(|u_k|^2)\,dx \to \infty \quad \text{as } k \to \infty.
\]
Therefore, \eqref{eq:sharp-critical-coefficient} holds for $N \geq 2$.
\end{proof}

We now prove that the linear dependence on $N$ in the endpoint
inequality is optimal. Recall the quantity $K_N$ from
\eqref{eq:KN-definition}.

\begin{lemma}[Optimal linear growth in the rank]
\label{lem:sharp-linear}
There are universal constants $c_1,c_2>0$ such that for any $N \geq 1$,
\begin{equation}  \label{eq:sharp-linear}
 c_1 N \leq K_N \leq c_2N,
\end{equation}
where $K_N$ is defined in \eqref{eq:KN-definition}.
\end{lemma}
\begin{proof}
The upper bound in \eqref{eq:sharp-linear} follows from
Lemma~\ref{lem:ruf-endpoint}. For the lower bound, take  a nonzero real-valued
$\psi\in C_c^\infty(B_1)$ and put
\[
 \varphi:=\frac{\psi}{\|\psi\|_{H^1}},
 \quad
 c_1:=\int_{\R^2}\Theta_{4\pi}(|\varphi(x)|^2)\,dx.
\]
Then $\|\varphi\|_{H^1}=1$. Moreover, $0<c_1<\infty$. Indeed, the finiteness of $c_1$ follows from the boundedness and compact support of $\varphi$, while the
positivity follows from $\Theta_{4\pi}(s)>0$ for $s>0$ and
$\varphi\not\equiv0$.

Choose $x_n:=(3n,0)$ and define
$u_n(x):=\varphi(x-x_n)$, $1\le n\le N$. Since
$\operatorname{supp}\varphi \subset B_1$, the
supports of the functions $u_n$ are pairwise disjoint. Therefore, the translation invariance of the $H^1$ norm and the disjointness imply
$\langle u_n,u_m\rangle_{H^1} =\delta_{nm}$ for $1 \le n,m\le N$, meaning $\{u_n\}_{n=1}^N$ is $H^1$-orthonormal. Since at a.e.\ point at most one $u_n$ is nonzero, and $\Theta_{4\pi}(0)=0$, we have
\begin{align*}
 \int_{\R^2}\Theta_{4\pi}\!\left(\sum_{n=1}^N|u_n(x)|^2\right)\,dx
 =\sum_{n=1}^N\int_{\R^2}\Theta_{4\pi}(|u_n(x)|^2)\,dx
 =N\int_{\R^2}\Theta_{4\pi}(|\varphi(x)|^2)\,dx=c_1 N.
\end{align*}
Consequently $K_N\ge c_1 N$, which completes the proof.
\end{proof}

\begin{proof}[\textbf{Proof of Theorem~\ref{thm:intro-whole-space}}]
The endpoint estimate~\eqref{eq:rank-one-defect} follows from Lemma \ref{lem:endpoint}.
Lemma~\ref{lem:sharp-critical-coefficient} gives
\eqref{eq:intro-alpha-sharp} for every $\alpha>4\pi$ and every
$N\ge1$. Finally, Lemma~\ref{lem:sharp-linear} yields ~\eqref{eq:KN-two-sided}.
\end{proof}

\subsection{Subcritical rank--mass inequality}
In this subsection, we establish a subcritical Moser--Trudinger type inequality.

\begin{lemma}
\label{lem:subcritical-interpolation}
There is a universal constant $C>0$ such that, for any
$0<\varepsilon\le1/2$, we have
\begin{equation}  \label{eq:subcritical-pointwise-interpolation}
 e^{4\pi(1-\varepsilon)s}-1
 \le \frac{C}{\varepsilon}\Theta_{4\pi}(s)^{1-\frac{\varepsilon}{2}}s^{\frac{\varepsilon}{2}}, \quad \forall s \geq 0.
\end{equation}
At $s=0$ the right side is understood to be zero.
\end{lemma}
\begin{proof}
Put $t=4\pi s$ and
\[
T(t):=\frac{e^t-1-t}{t},\quad t>0.
\]
Since $\Theta_{4\pi}(s)=4\pi T(t)$ and $s=t/(4\pi)$,
 \[
\Theta_{4\pi}(s)^{1-\varepsilon/2}s^{\varepsilon/2}
=(4\pi)^{1-\varepsilon}T(t)^{1-\varepsilon/2}t^{\varepsilon/2}.
\]
The factor $(4\pi)^{1-\varepsilon}$ is bounded above and below by positive universal constants for $0<\varepsilon\le1/2$. Thus it suffices to prove, with a possibly different universal constant,
\begin{equation}  \label{eq:T-interpolation}
 e^{(1-\varepsilon)t}-1
 \leq \frac{C}{\varepsilon}  T(t)^{1-\frac{\varepsilon}{2}}t^{\frac{\varepsilon}{2}}.
\end{equation}
For $0<t\le2$, by Taylor's formula, we have
$T(t)\ge t/2$, whereas
$e^{(1-\varepsilon)t}-1\le e^2t$.  Hence
\eqref{eq:T-interpolation} holds in this range, even without the
factor $\varepsilon^{-1}$.

For $t\ge2$, there is $c>0$ such that $T(t)\ge ce^t/t$.  Consequently, 
\[
 \frac{e^{(1-\varepsilon)t}-1}{T(t)^{1-\frac{\varepsilon}{2}}t^{\frac{\varepsilon}{2}}} \leq C e^{-\frac{\varepsilon t}{2}}t^{1-\varepsilon} \leq C\sup_{t>0}(e^{-\frac{\varepsilon t}{2}} t^{1-\varepsilon})
  = C \left(\frac{2(1-\varepsilon)}{e\varepsilon}\right)^{1-\varepsilon}
  \le\frac{C}{\varepsilon}.
\]
This implies~\eqref{eq:T-interpolation}. The proof is complete.
\end{proof}

 \begin{corollary}[Subcritical rank-mass Moser--Trudinger type inequality]
\label{cor:ruf-subcritical}
For every $H^1$-orthonormal family $\{u_n\}_{n=1}^N$, and $0<\varepsilon \leq \frac{1}{2}$, one has
\begin{equation} \label{eq:ruf-subcritical}
\int_{\R^2}\Big(e^{4\pi(1-\varepsilon)\rho(x)}-1\Big)\,dx
\le \frac{C}{\varepsilon}
 N^{1-\frac{\varepsilon}{2}}
 \left(\int_{\R^2}\rho\,dx\right)^{\frac{\varepsilon}{2}},
\end{equation}
where $C>0$ is a universal constant.
When $N=1$, every $u \in H^1(\R^2)$ with $\| u \|_{H^1} = 1$ satisfies
\begin{equation}  \label{eq:scalar-subcritical-mass}
 \int_{\R^2}\Big(e^{4\pi(1-\varepsilon)|u|^2}-1\Big)\,dx
 \le\frac{C}{\varepsilon}\|u\|_{L^2}^2.
\end{equation}
\end{corollary}

\begin{proof}
By Lemma~\ref{lem:subcritical-interpolation},
H\"older's inequality, and Lemma~\ref{lem:ruf-endpoint}, we obtain
\begin{align*}
 \int_{\R^2}(e^{4\pi(1-\varepsilon)\rho}-1)\,dx
 &\le\frac{C}{\varepsilon}
      \int_{\R^2}\Theta_{4\pi}(\rho)^{1-\frac{\varepsilon}{2}}\rho^{\frac{\varepsilon}{2}}\,dx\\
 &\le\frac{C}{\varepsilon}
      \left(\int_{\R^2}\Theta_{4\pi}(\rho)\,dx\right)^{1-\frac{\varepsilon}{2}}
      \left(\int_{\R^2}\rho\,dx\right)^{\frac{\varepsilon}{2}} \leq \frac{C}{\varepsilon} N^{1-\frac{\varepsilon}{2}}
      \left(\int_{\R^2}\rho\,dx\right)^{\frac{\varepsilon}{2}},
\end{align*}
which is~\eqref{eq:ruf-subcritical}.

When $N=1$, applying Lemma~\ref{lem:subcritical-interpolation}, H\"older's inequality and Proposition~\ref{prop:scalar-exact-growth}, we obtain
 \[
 \int_{\R^2}(e^{4\pi(1-\varepsilon)|u|^2}-1)\,dx
 \le\frac{C}{\varepsilon}(C_{\rm sc}\|u\|_{L^2}^2)^{1-\varepsilon/2}\|u\|_{L^2}^{\varepsilon}
 \le\frac{C'}{\varepsilon} \|u\|_{L^2}^2,
\]
which proves~\eqref{eq:scalar-subcritical-mass}. 
\end{proof}

\begin{remark}
Since $H^1$-orthonormality gives $\int_{\R^2}\rho\,dx\le N$, the mixed
rank--mass factor in~\eqref{eq:ruf-subcritical} satisfies
$N^{1-\varepsilon/2}(\int_{\R^2}\rho\,dx)^{\varepsilon/2}\le N$.
Thus the estimate \eqref{eq:ruf-subcritical} implies a rank-only linear bound. The spectral test
\eqref{eq:subcritical-mass-obstruction} shows that its rank
contribution cannot be replaced by the mass alone.
\end{remark}

\section{Dirichlet families} \label{sec:dirichlet-families}
\subsection{Construction of an $H^1$-orthonormal family}

Recall that $\mu_j$ and $\psi_j$ denote the eigenvalues and corresponding eigenfunctions of the Dirichlet Laplacian $-\Delta$ in the unit ball $B_1$. Moreover, $\{\psi_j\}_{j\ge1}\subset H_0^1(B_1)$ forms an orthonormal basis of $L^2(B_1)$. Thus for any $R>0$, the Dirichlet operator
$\cA_R:=-\Delta+1$ on  $L^2(B_R)$ has eigenvalues
\begin{equation}\label{eq:lambdaj}
 \lambda_j(R):=1+\frac{\mu_j}{R^2},\qquad j\in\N.
\end{equation}

By the classical Dirichlet Weyl law~\cite{Weyl1912} on the unit ball,
 \begin{equation} \label{eq:weyl-ball}
\mu_j =\ 4j\,(1+o(1)_{j \to \infty}).
\end{equation}

\begin{lemma}[$H^1$-orthonormal family from Dirichlet truncation]
\label{lem:dirichlet-family}
For $N\ge1$ and $R>0$, let $\{u_j^{(R)}\}_{j=1}^N$ and
$\rho_{N,R}$ be defined by~\eqref{u_j^(R)} and~\eqref{rho_NR}, respectively. Then the
functions $u_j^{(R)}$ belong to $H^1(\R^2)$ and form an
$H^1$-orthonormal  family. The density $\rho_{N,R}$ is continuous,
is supported in $\overline{B_R}$, and satisfies
\begin{equation} \label{eq:rhoN-dirichlet}
\int_{\R^2}\rho_{N,R}(x)\,dx\ =\ \sum_{j=1}^N \lambda_j(R)^{-1}.
\end{equation}
\end{lemma}
\begin{proof}
Since  $\psi_j\in H_0^1(B_1)$, $\psi_j(\cdot/R)\in H_0^1(B_R)$, and hence $u_j^{(R)}\in H^1(\R^2)$.
It is well-known that  $\psi_j\in C(\overline{B_1})$ with
$\psi_j|_{\partial B_1}=0$. Therefore $u_j^{(R)}$, and hence $\rho_{N,R}$, is continuous on $\R^2$ and supported in $\overline{B_R}$.

Using~ \eqref{u_j^(R)}, the self-adjointness of $\cA_R$, the eigenvalue
equation  $\cA_R \psi_j(\cdot/R)=
\lambda_j(R)\psi_j(\cdot/R)$, and orthonormality, we obtain
 \begin{align*}
\langle u_i^{(R)},u_j^{(R)}\rangle_{H^1}
&=
\lambda_i(R)^{-1/2} \lambda_j(R)^{-1/2}R^{-2}
\big\langle \psi_i(\cdot/R),
 \cA_R \psi_j(\cdot/R)\big\rangle_{L^2(B_R)}\\
&=
\lambda_i(R)^{-1/2} \lambda_j(R)^{-1/2}
\lambda_j(R) \delta_{ij}
=\delta_{ij}.
\end{align*}
Hence $\{u_j^{(R)}\}_{j=1}^N$ is $H^1$-orthonormal. From  \eqref{u_j^(R)}, we also have
 \begin{align*}
\int_{\R^2}\rho_{N,R}\,dx
&=\sum_{j=1}^N \lambda_j(R)^{-1}R^{-2}
  \int_{B_R}|\psi_j(x/R)|^2\,dx=\sum_{j=1}^N \lambda_j(R)^{-1}
  \|\psi_j\|_{L^2(B_1)}^2
= \sum_{j=1}^N \lambda_j(R)^{-1},
\end{align*}
which is~\eqref{eq:rhoN-dirichlet}.
\end{proof}

\subsection{Spectral estimates}

 For $y\in\overline{B_1}$ and $\Lambda\ge0$, define
\begin{equation}\label{def:N_y(mu)}
 N_y(\Lambda):=\sum_{j:\mu_j\le\Lambda}|\psi_j(y)|^2.
\end{equation}
The sum includes every eigenfunction with eigenvalue at most $\Lambda$,
counted with multiplicity, and is independent of the basis chosen within
each eigenspace. For each integer $k\in\{0,\ldots,N\}$, define
\begin{equation}\label{def:S_{N-k_N,N}(y)}
 S_{k,N}(y):=\sum_{j=k+1}^N|\psi_j(y)|^2;
\end{equation}
the empty sum for $k=N$ is zero. A choice of a partial eigenspace may
change $S_{0,N}$, but its contribution is bounded by the corresponding
full spectral projector. The next  estimates apply in both the
intermediate-density and low-density regimes.

\begin{lemma}
\label{lem:sublinear-cluster}
Define the sequences $\{\mu_N^{<}\}_{N\ge1}$ and $\{k_N\}_{N\ge1}$ by
 \[
\mu_N^{<}=\sup\Big(\{0\}\cup\big\{\mu_j\mid\mu_j<\mu_N\big\}\Big),
\quad
k_N=\#\big\{1\le j\le N\mid\mu_j=\mu_N\big\}.
\]
Then
 \begin{equation} \label{est:gap}
k_N=o(N)_{N\to\infty}
\quad\text{and}\quad
\mu_N-\mu_N^{<}=o(N)_{N\to\infty}.
\end{equation}
\end{lemma}

\begin{proof}
 Define the sequence $\{j_0(N)\}_{N\ge1}$ by $j_0(N)=N-k_N+1$. Since $\mu_{j_0(N)}=\mu_N\to\infty$,
it follows that $j_0(N)\to\infty$. By Weyl's law, we have
\[
 \frac{j_0(N)}{N}
 =\frac{\mu_N/N}{\mu_{j_0(N)}/j_0(N)}\to1.
\]
 Thus $k_N/N=1-j_0(N)/N+1/N \to 0$, which implies the first equality in \eqref{est:gap}.
 
 Next, for all sufficiently large $N$, we have  $j_0(N)\ge2$ and
$\mu_N^{<}=\mu_{j_0(N)-1}$. Therefore
\[
 \frac{\mu_N^{<}}{N}
 =\frac{\mu_{j_0(N)-1}}{j_0(N)-1}\frac{j_0(N)-1}{N}
 \to4.
\]
 Subtracting this from $\mu_N/N\to4$ implies the second equality in \eqref{est:gap}.
\end{proof}

The following local spectral estimate from
\cite[Theorem~5.1, formula~(5.3)]{Hormander68} will be useful in the sequel.

\begin{proposition}[H\"ormander's local spectral
estimate]
\label{prop:hormander-local-weyl}
Let $\Omega\subset\R^2$ be a bounded domain with smooth boundary and $K\Subset\Omega$ be compact. Let
$\{\mu_{\Omega,j}\}_{j\ge1}$ be the eigenvalues of the Dirichlet Laplacian $-\Delta$ in $\Omega$ with
corresponding $L^2(\Omega)$-orthonormal eigenfunctions $\{\psi_{\Omega,j}\}_{j \geq 1}$.  Then there is $C=C(\Omega,K)<\infty$ such that for any $\Lambda \geq 1$ and any $y \in K$,
\[
\Big|\sum_{j\,\mid\,\mu_{\Omega,j} \le\Lambda}|\psi_{\Omega,j}(y)|^2-\frac{\Lambda}{4\pi}\Big|
\ \le\
C\sqrt\Lambda.
\]
\end{proposition}

\section{Intermediate-density regime}
\label{sec:intermediate-density}

We  study the Dirichlet family introduced in~\eqref{u_j^(R)} and
\eqref{rho_NR}. We first establish the mass limit and derive the obstruction to mass-only bounds. Then spectral
and heat-kernel estimates yield the interior density and energy limits. Throughout this section, let $\{\kappa_N\}_{N\ge1}$ be a sequence of positive numbers such that $\kappa_N\to\kappa\in(0,\infty)$ and set $R_N=\sqrt{\kappa_N N}$.
Let $\{u_j^{(R_N)}\}_{j=1}^N$ and $\rho_{N,R_N}$ be defined by
\eqref{u_j^(R)} and~\eqref{rho_NR}, respectively.

\subsection{Mass-per-state limit}
Define
\begin{equation}\label{def:g(kappa)}
 g(\kappa):=\frac{\kappa}{4}\ln\left(1+\frac{4}{\kappa}\right), \quad \kappa>0.
\end{equation}
The following limit concerns the mass per state.

 \begin{lemma}
\label{lem:density-family}
We have
\begin{equation}\label{eq:density-limit}
 \lim_{N\to\infty}\frac{1}N\int_{\R^2}\rho_{N,R_N}(x)\,dx
 =\lim_{N\to\infty}\frac{1}N\sum_{j=1}^N\lambda_j(R_N)^{-1}
 =g(\kappa).
\end{equation}
\end{lemma}
 \begin{proof}
 By the mass identity~\eqref{eq:rhoN-dirichlet} and~\eqref{eq:lambdaj}, we derive
\[
 \frac{1}N\int_{\R^2}\rho_{N,R_N}\,dx
 =\frac{1}N\sum_{j=1}^N\frac{\kappa_N}{\mu_j/N+\kappa_N}.
\]
 Fix $\eta\in(0,4)$.  By Weyl's law, there exists $J_0=J_0(\eta)$ such that
$(4-\eta)j\le\mu_j\le(4+\eta)j$ for every $j\ge J_0$.
 For $N\ge J_0$, the initial sum satisfies
\[
 0\le\frac{1}N\sum_{j=1}^{J_0-1}\lambda_j(R_N)^{-1}
 \le\frac{J_0-1}{N} \to 0 \quad \text{as } N \to \infty.
\]
 Define, for $t\in[0,1]$,
\[
 h_{\eta,N}^{\pm}(t):=\frac{\kappa_N}{(4\pm\eta)t+\kappa_N},
 \qquad
 h_\eta^{\pm}(t):=\frac{\kappa}{(4\pm\eta)t+\kappa}.
\]
 For $J_0\le j\le N$,
$h_{\eta,N}^{+}(j/N)\le\lambda_j(R_N)^{-1}
\le h_{\eta,N}^{-}(j/N)$.
Since  $\kappa_N\ge\kappa/2$ for all sufficiently large $N$, we have
\[
 \sup_{t\in[0,1]}|h_{\eta,N}^{\pm}(t)-h_\eta^{\pm}(t)|
 \le\frac{2(4\pm\eta)}{\kappa^2}|\kappa_N-\kappa|
 \to 0 \quad \text{as } N \to \infty.
\]
Therefore, the Riemann sums converge to the integrals of
$h_\eta^{\pm}$, and removing their first $J_0-1$ terms changes them
by at most $(J_0-1)/N$. Consequently,
\[
 \int_0^1\frac{\kappa\,dt}{(4+\eta)t+\kappa}
 \le\liminf_{N\to\infty}\frac{1}N\sum_{j=1}^N\lambda_j(R_N)^{-1}
 \le\limsup_{N\to\infty}\frac{1}N\sum_{j=1}^N\lambda_j(R_N)^{-1}
 \le\int_0^1\frac{\kappa\,dt}{(4-\eta)t+\kappa}.
\]
By letting $\eta \downarrow 0$, we derive
$\int_0^1\kappa(4t+\kappa)^{-1}\,dt=g(\kappa)$ and hence obtain the claim.
\end{proof}

Define 
\begin{equation}\label{eq:rho-infty-def}
 \rho_\infty(\kappa):=\frac{1}{4\pi}\ln\left(1+\frac{4}{\kappa}\right), \quad \kappa>0.
\end{equation}
Then we have the relation $g(\kappa)=\pi\kappa\rho_\infty(\kappa)$.

\subsection{Obstruction to mass-only bounds}

\begin{proposition}\label{prop:no-rank-uniform-mass}
There is no universal constant $C>0$ such that for every
$H^1$-orthonormal family $\{u_n\}_{n=1}^N$ ($N\ge1$) with density
$\rho$ there holds
\begin{equation}\label{eq:false-uniform-mass}
 \int_{\R^2}\Theta_{4\pi}(\rho)\,dx
 \le C\int_{\R^2}\rho\,dx.
\end{equation}
\end{proposition}
\begin{proof}
Suppose for a contradiction that such a constant exists. Fix $\kappa>0$ and let $\{\kappa_N\}_{N \geq 1}$ be any sequence of positive numbers converging to $\kappa$. Set
$R_N=\sqrt{\kappa_N N}$ and $\beta=\ln(1+4/\kappa)$.
By Lemma~\ref{lem:density-family}, we have
\[
 \lim_{N \to \infty}\frac{1}{|B_{R_N}|}\int_{B_{R_N}}\rho_{N,R_N}\,dx
 =\rho_\infty(\kappa)=\frac{\beta}{4\pi}.
\]
By Jensen's inequality, we obtain
\begin{align*}
 C \ge\frac{\int_{\R^2}\Theta_{4\pi}(\rho_{N,R_N})\,dx}
               {\int_{\R^2}\rho_{N,R_N}\,dx} \ge\frac{\displaystyle\Theta_{4\pi}\!\left(
       \frac{1}{|B_{R_N}|}\int_{B_{R_N}}\rho_{N,R_N}\,dx\right)}
             {\displaystyle\frac{1}{|B_{R_N}|}
                \int_{B_{R_N}}\rho_{N,R_N}\,dx}.
\end{align*}
Letting $N\to\infty$ yields
$$
 C\ge\frac{\Theta_{4\pi}(\rho_\infty(\kappa))}{\rho_\infty(\kappa)}
 =\frac{16\pi^2(4-\kappa\beta)}{\kappa\beta^2}.
$$
Hence letting $\kappa\downarrow0$ implies a contradiction because the right-hand side tends to infinity. Thus we obtain the claim.
\end{proof}

 \begin{remark}
The same argument rules out a rank-uniform mass-only bound for
$\int_{\R^2}(e^{4\pi(1-\varepsilon)\rho}-1)\,dx$ for every fixed
$0<\varepsilon<1$. In particular, the  right-hand side of
Corollary~\ref{cor:ruf-subcritical} cannot be replaced by
$C_\varepsilon\int_{\R^2}\rho\,dx$.
Indeed, define the convex function
$\Xi_\varepsilon(s):=e^{4\pi(1-\varepsilon)s}-1$ for $s\ge0$.
With the notation of the preceding proof, Jensen's inequality yields
\begin{equation}\label{eq:subcritical-mass-obstruction}
 \liminf_{N\to\infty}
 \frac{\int_{\R^2}\Xi_\varepsilon(\rho_{N,R_N})\,dx}{\int_{\R^2}\rho_{N,R_N}\,dx}
 \ge\frac{\Xi_\varepsilon(\rho_\infty(\kappa))}{\rho_\infty(\kappa)}
 =\frac{4\pi(e^{(1-\varepsilon)\beta}-1)}{\beta}.
\end{equation}
For every fixed $0<\varepsilon<1$, the last expression diverges as
$\kappa\downarrow0$, since it is asymptotic to
$4\pi(4/\kappa)^{1-\varepsilon}/\beta$.
\end{remark}

\subsection{Density bounds}
We will apply Proposition \ref{prop:hormander-local-weyl} with $\Omega=B_1$ and $K=\overline{B_{1-\delta}}$ for fixed
$\delta\in(0,1)$. We will also need the following classical estimate, which can be found in, e.g., \cite[eq.~(1.9.1), p.~52]{Davies89}.

\begin{lemma}
\label{lem:dirichlet-monotone}
Let $\Omega\subset\R^2$ be bounded and open, let
$k_{-\Delta}^\Omega(t,x,y)$ be the heat kernel of the Dirichlet
Laplacian in $\Omega$, and let
$k_{-\Delta}(t,x,y)=(4\pi t)^{-1}e^{-|x-y|^2/(4t)}$ be the heat kernel
on $\R^2$. Then, 
$k_{-\Delta}^{\Omega}(t,x,x)\ \le\ \frac{1}{4\pi t}$ 
for a.e. $x\in\Omega$ and $t>0$.
\end{lemma}

\begin{lemma}[Global density bound]
\label{lem:uniform-rhoN}
There is a positive constant $\tilde C$, depending on the sequence $\{\kappa_N\}_{N \geq 1}$ but not on $N$ or the choice of eigenfunctions, such that
$\rho_{N,R_N}(x)\le \tilde C$ for every $N\ge1$ and every $x\in\R^2$.
\end{lemma}
\begin{proof}
Applying Lemma~\ref{lem:bessel} to the first $N$ normalized eigenfunctions of the Dirichlet operator $\cA_{R_N}=-\Delta+1$ on $B_{R_N}$, we have
for every $t>0$ and a.e. $x\in B_{R_N}$, 
\[
\sum_{j=1}^N e^{-2t \lambda_j(R_N)} \big|u_j^{(R_N)}(x)\big|^2
\ \le\
Q_t^{R_N}(x):=(\cA_{R_N}^{-1}e^{-2t \cA_{R_N}})(x,x)
 =\int_{2t}^\infty e^{-s}k_{-\Delta}^{B_{R_N}}(s,x,x)\,ds.
\]
From Lemma~\ref{lem:dirichlet-monotone} with $\Omega=B_{R_N}$ and
Lemma~\ref{lem:E1-bound}, we infer that, for $t>0$ and a.e.\ $x\in B_{R_N}$,
\begin{equation} \label{eq:Q-dirichlet}
Q_t^{R_N}(x)
\ \le\
\int_{2t}^\infty\frac{e^{-s}}{4\pi s}\,ds
=
\frac{1}{4\pi}\Ei(2t)
\ \le\
\frac{1}{4\pi}\ln\Big(1+\frac{1}{2t}\Big).
\end{equation}
Since
$u_j^{(R_N)}$ is an eigenfunction of $\cA_{R_N}$ with eigenvalue
$\lambda_j(R_N)$, and $\lambda_j(R_N)\le\lambda_N(R_N)$ for $j\le N$, we deduce
$e^{-2t\lambda_N(R_N)}\rho_{N,R_N}(x)\le Q_t^{R_N}(x)$. Taking
$t=1/2$ and using \eqref{eq:Q-dirichlet} yield
$\rho_{N,R_N}(x)\ \le\ e^{\lambda_N(R_N)} \, \frac{\ln 2}{4\pi}$
for a.e. $x\in\R^2$.
Because $\rho_{N,R_N}$ is continuous by
Lemma~\ref{lem:dirichlet-family}, the same inequality holds for every
$x\in\R^2$.
By Weyl's law~\eqref{eq:weyl-ball},
$\lambda_N(R_N)=1+\mu_N/R_N^2\to1+4/\kappa$ as $N\to\infty$
with $R_N=\sqrt{\kappa_N N}$. Hence the sequence $\{\lambda_N(R_N)\}_{N\ge1}$ is bounded. Put $\tilde C:=\frac{\ln 2}{4\pi}\sup_{N\ge1}e^{\lambda_N(R_N)}<\infty$. This constant may depend on the entire sequence $\{\kappa_N\}_{N \geq 1}$, including the initial terms. Thus we obtain the desired bound.
\end{proof}

\begin{lemma}[Spectral truncation]
\label{lem:dilute-projector}
There is a positive constant $C_1$, independent of $N$ and of the
choice of eigenfunctions within each eigenspace, such that for any $N \geq 1$,
\begin{equation}\label{eq:dilute-projector-global}
 \frac{S_{0,N}(y)}N\le\frac{e\, \mu_N}{4\pi N}\le C_1
 \quad \text{for any } y \in \overline{B_1}.
\end{equation}
For every fixed $\delta\in(0,1)$, there is a sequence
$\{\varepsilon_{N,\delta}\}_{N\ge1}$ converging to $0$ as $N \to \infty$ 
such that
\begin{equation}\label{eq:dilute-projector-bulk}
 \sup_{y\in\overline{B_{1-\delta}}}
 \left|\frac{S_{0,N}(y)}N-\frac{1}{\pi}\right|
 \le\varepsilon_{N,\delta}.
\end{equation}
\end{lemma}

\begin{proof}
For $t>0$, by the spectral expansion of the Dirichlet heat kernel and
Lemma~\ref{lem:dirichlet-monotone}, we obtain
\[
S_{0,N}(y)
 \leq  e^{t\mu_N}\sum_{j=1}^{\infty}e^{-t\mu_j}
       |\psi_j(y)|^2
 =  e^{t\mu_N} k_{-\Delta}^{B_1}(t,y,y)\le\frac{ e^{t\mu_N}}{4\pi t} \quad \text{for a.e. } y \in B_1. 
\]
Taking $t=\mu_N^{-1}$ and using continuity
of $S_{0,N}$, with its zero boundary values, imply the first bound
in~\eqref{eq:dilute-projector-global} everywhere on $\overline{B_1}$.
The second estimate in ~\eqref{eq:dilute-projector-global} follows from~\eqref{eq:weyl-ball} since
$\sup_N\mu_N/N<\infty$.

Let $N_y(\Lambda)$, $\Lambda>0$, be defined in \eqref{def:N_y(mu)}, and let $\mu_N^{<}$ and $k_N$ be defined in Lemma~\ref{lem:sublinear-cluster}. 
Then we have
\begin{equation} \label{est:S_0}
 S_{0,N}(y)=N_y(\mu_N^{<})+S_{N-k_N,N}(y),\qquad
 0\le S_{N-k_N,N}(y)\le N_y(\mu_N)-N_y(\mu_N^{<}).
\end{equation}
Here $S_{N-k_N,N}$ is the one-body density of the selected part of the eigenspace associated with $\mu_N$.
The last inequality in \eqref{est:S_0} holds for every orthonormal choice in that eigenspace, by completing it to an orthonormal basis. In particular, from \eqref{est:S_0}, we have
\begin{equation} \label{est1:S0} 
N_y(\mu_N^{<})\le S_{0,N}(y)\le N_y(\mu_N), \quad \forall y \in \overline{B_1}.
\end{equation}
Fix $\delta\in(0,1)$. This, together with  Proposition~\ref{prop:hormander-local-weyl} with $\Omega=B_1$ and $K=\overline{B_{1-\delta}}$ and~\eqref{est1:S0}, implies
\[
\left( \frac{\mu_N^{<}}{4\pi N} - \frac{1}{\pi} \right) - \frac{C_\delta(1 + \sqrt{\mu_N})}{N} \le \frac{S_{0,N}(y)}{N} - \frac{1}{\pi} \le \left( \frac{\mu_N}{4\pi N} - \frac{1}{\pi} \right) + \frac{C_\delta(1 + \sqrt{\mu_N})}{N},
\]
for some positive constant $C_\delta$ independent of $N$. Therefore,
\begin{equation} \label{est:S/N-pi}
\left | \frac{S_{0,N}(y)}{N} - \frac{1}{\pi} \right | \le \left | \frac{\mu_N}{4\pi N} - \frac{1}{\pi} \right | + \frac{\mu_N - \mu_N^{<}}{4\pi N} + \frac{C_\delta(1 + \sqrt{\mu_N})}{N} =: \varepsilon_{N,\delta}, \quad \forall y \in \overline{B_{1-\delta}},
\end{equation}
which implies \eqref{eq:dilute-projector-bulk}. 
Since  $\mu_N=4N(1+ o(1)_{N \to \infty})$ (due to ~\eqref{eq:weyl-ball}) and 
$\mu_N-\mu_N^{<}=o(N)_{N \to \infty}$ (due to
Lemma~\ref{lem:sublinear-cluster}), we derive that $\varepsilon_{N,\delta} \to 0$ as $N \to \infty$. The proof is complete.
\end{proof}

Next we will prove uniform convergence to $\rho_\infty(\kappa)$ on each fixed relative interior ball, where $\rho_\infty$ is defined in \eqref{eq:rho-infty-def}.
Together with the global density bound, this will give the limit of
$N^{-1}\int_{\R^2}\Theta_{4\pi}(\rho_{N,R_N})\,dx$.

\begin{lemma}[Uniform interior limit]
\label{lem:pointwise-bulk}
 For any $\delta\in(0,1)$, we have
 \begin{equation} \label{eq:bulk-limit}
\lim_{N\to\infty}\ \sup_{x\in B_{(1-\delta)R_N}}\big|\rho_{N,R_N}(x)-\rho_\infty(\kappa)\big|\ =\ 0.
\end{equation}
\end{lemma}
 \begin{proof}
For  $x\in B_{(1-\delta)R_N}$, put  $y=x/R_N\in\overline{B_{1-\delta}}$.
Note that the function $N_y$ from~\eqref{def:N_y(mu)} is nondecreasing
and right-continuous on $[0,\infty)$, with $N_y(0)=0$.
Let  $\mu_N^{<}$ and $k_N$ be as in Lemma~\ref{lem:sublinear-cluster}.
\smallskip

\noindent\textbf{Step 1:} We will separate the top cluster. We note that the eigenvalues
 $\mu_1 \leq \cdots \leq \mu_N$ are listed in  nondecreasing order and exactly $k_N$ of them (necessarily the last $k_N$ indices $N-k_N+1,\dots,N$) equal  $\mu_N$, while $\mu_j<\mu_N$ for $j\le N-k_N$. Moreover, there is no eigenvalue in  $(\mu_N^{<},\mu_N)$.
Set 
$$
w_N(\mu):=(1+\mu/R_N^2)^{-1}.
$$ 
The sum $S_{N-k_N,N}$
from~\eqref{def:S_{N-k_N,N}(y)} is the density of the selected part of the top
eigenspace. The density decomposes as follows.
From \eqref{rho_NR}, \eqref{u_j^(R)} and \eqref{eq:lambdaj}, together with the definition of  $w_N$ and the fact that  $\{j:\mu_j\le\mu_N^{<}\}=\{1,\dots,N-k_N\}$, we infer
 \begin{equation}
\begin{aligned}
\rho_{N,R_N}(x) =R_N^{-2}\sum_{j=1}^N w_N(\mu_j)|\psi_j(y)|^2 =\underbrace{R_N^{-2}\int_{[0,\mu_N^{<}]}w_N\,dN_y(\mu)}_{=:I_N(y)}
+\underbrace{R_N^{-2}w_N(\mu_N)S_{N-k_N,N}(y)}_{=:\mathrm{II}_N(y)}.
\end{aligned}
\label{eq:split-cluster}
\end{equation}

\noindent\textbf{Step 2:} We will prove that the correction term
$\mathrm{II}_N(y)$ tends to zero uniformly in the interior. Since  $\psi_{N-k_N+1},\dots,\psi_N$ are orthonormal and belong to
 $E_{\mu_N}:=\ker(-\Delta_{B_1}-\mu_N)$, we can extend this orthonormal system to an orthonormal basis $\{e_1,\dots,e_d\}$ of  $E_{\mu_N}$. After relabeling so that the selected functions are $e_1,\dots,e_{k_N}$, by using the elementary  inequality
$\sum_{i=1}^{k_N}|e_i(y)|^2\le\sum_{i=1}^d|e_i(y)|^2$, we have
 \[
0\ \le\ S_{N-k_N,N}(y)\ \le\ N_y(\mu_N)-N_y(\mu_N^{<})
\quad\text{for every }y\in B_{1-\delta}.
\]
Note that  $\mu_N\to\infty$ by~\eqref{eq:weyl-ball} and
 $\mu_N^{<}\to\infty$ by Lemma~\ref{lem:sublinear-cluster}. Applying
Proposition~\ref{prop:hormander-local-weyl} at  $\mu_N$ and $\mu_N^{<}$, with the constant $C=C(\delta)$, gives
 \begin{align*}
N_y(\mu_N)-N_y(\mu_N^{<})
\leq \frac{\mu_N-\mu_N^{<}}{4\pi}
 +C\big(\sqrt{\mu_N}+\sqrt{\mu_N^{<}}\big)  \leq \frac{\mu_N-\mu_N^{<}}{4\pi}+2C\sqrt{\mu_N} = o(N)_{N\to\infty},
\end{align*}
uniformly for $y\in B_{1-\delta}$. The last equality above holds because  $\mu_N-\mu_N^{<}=o(N)_{N\to\infty}$ (by Lemma~\ref{lem:sublinear-cluster}) and  $\mu_N=O(N)_{N\to\infty}$ (by~\eqref{eq:weyl-ball}). Note that the term $o(N)_{N\to\infty}$ above is independent of $y$. Since  $0\le w_N\le1$, $R_N^2=\kappa_N N$ and $\kappa_N \geq \kappa/2$ for sufficiently large $N$, 
\begin{equation}
\begin{aligned}
0\le \mathrm{II}_N(y)
\le R_N^{-2}\big(N_y(\mu_N)-N_y(\mu_N^{<})\big)
=o(1)_{N\to\infty}
\quad\text{uniformly for }y\in B_{1-\delta}.
\end{aligned}
\label{eq:correction-vanishes}
\end{equation}

\noindent\textbf{Step 3:} We will prove an asymptotic formula for the leading term $I_N(y)$.
By integration by parts ($w_N$ is $C^1$ and decreasing, $N_y(0)=0$),
 \begin{equation}
\int_{[0,\mu_N^{<}]}w_N\,dN_y(\mu)
=
w_N(\mu_N^{<})N_y(\mu_N^{<})+R_N^{-2}\int_0^{\mu_N^{<}}\frac{N_y(\mu)}{(1+\mu/R_N^2)^2}\,d\mu.
\label{eq:ibp-bulk}
\end{equation}
By applying Proposition~\ref{prop:hormander-local-weyl} for $\overline{B_{1-\delta}}$, we have
$N_y(\mu)=\mu/(4\pi)+E_y(\mu)$ with $|E_y(\mu)|\le C(1+\sqrt\mu)$ for
all $\mu\ge0$ and $y\in\overline{B_{1-\delta}}$, where $C$ depends only
on $\delta$ (and absorbs the bound needed on $\mu\in[0,1]$, where
$N_y(\mu)$ is trivially bounded). Substituting into~\eqref{eq:ibp-bulk}
and multiplying by  $R_N^{-2}$ yield
\begin{equation}
\begin{aligned}
R_N^{-2}\int_{[0,\mu_N^{<}]}w_N\,dN_y(\mu) &= \left(R_N^{-2} w_N(\mu_N^{<})\frac{\mu_N^{<}}{4\pi}
+
R_N^{-4}\int_0^{\mu_N^{<}}\frac{\mu}{4\pi(1+\mu/R_N^2)^2}\,d\mu \right) \\
&\quad + \left( R_N^{-2} w_N(\mu_N^{<})E_y(\mu_N^{<})+R_N^{-4}\int_0^{\mu_N^{<}}\frac{E_y(\mu)}{(1+\mu/R_N^2)^2}\,d\mu \right) \\
&=: I_{N,1}(y) + I_{N,2}(y).
\end{aligned}
\end{equation}
Define the sequence $\{\ell_N\}_{N\ge1}$ by $\ell_N:=\mu_N^{<}/R_N^2$. By a change of variable, we obtain
\[
I_{N,1}(y)
 = \frac{1}{4\pi}\cdot\frac{\ell_N}{1+\ell_N}
 + \frac{1}{4\pi} \left( \ln(1+\ell_N)-\frac{\ell_N}{1+\ell_N} \right)
 = \frac{1}{4\pi} \ln(1+\ell_N).
\]
Next, using  $w_N(\mu_N^{<})\le1$,
$|E_y(\mu)|\le C(1+\sqrt\mu)$ and the change of variables
 $\mu=R_N^2t$, together with the identities $\int_0^\infty\frac{dt}{(1+t)^2}=1$ and $\int_0^\infty\frac{t^{1/2}}{(1+t)^2}\,dt=\frac{\pi}{2}$, we have
 \begin{equation} \label{est:Err}
|I_{N,2}(y)|
\leq CR_N^{-2}(1+\sqrt{\mu_N^{<}}) + CR_N^{-2}+\frac{C\pi}{2}R_N^{-1}
\leq C_{\kappa,\delta} N^{-\frac{1}{2}}.
\end{equation}
The last inequality holds for all sufficiently large $N$ since $R_N^2 = \kappa_N N \geq (\kappa N)/2$ and $\mu_N^{<} \leq \mu_N = O(N)_{N \to \infty}$. Here $C_{\kappa,\delta}$ is independent of
$y\in\overline{B_{1-\delta}}$. Hence
$\sup_{y\in\overline{B_{1-\delta}}}|I_{N,2}(y)|\to0$
as $N\to\infty$.

From the above expressions, we deduce that
\begin{equation} \label{est:I_N(y)} I_N(y)=\frac{1}{4\pi}\ln(1+\ell_N)+ I_{N,2}(y).
\end{equation}
By~\eqref{eq:weyl-ball},
 $ \frac{\mu_N}{R_N^2}
 =\frac{\mu_N}{\kappa_N N}\to  \frac{4}{\kappa}$ as $N \to \infty$. 
 Moreover, we infer from Lemma~\ref{lem:sublinear-cluster}, $R_N^2=\kappa_N N$ and $\kappa_N \to \kappa>0$ that
$(\mu_N-\mu_N^{<})/R_N^2=o(1)_{N \to \infty}$. Therefore
$\ell_N=\frac{\mu_N^{<}}{R_N^2}
 =\frac{\mu_N}{R_N^2}-\frac{\mu_N-\mu_N^{<}}{R_N^2}
 \to  \frac{4}{\kappa}$.
Consequently,
\begin{equation} \label{eq:ell-N-limit}
\frac{1}{4\pi}\ln(1+\ell_N) \to \frac{1}{4\pi} \ln(1+\frac{4}{\kappa})=\rho_\infty(\kappa) \quad \text{as } N \to \infty.
\end{equation}
Note that this limit does not depend on $y$. \medskip

\noindent\textbf{Step 4:} We will prove the claimed uniform convergence. Combining~\eqref{eq:split-cluster},
\eqref{est:I_N(y)}, \eqref{est:Err}, \eqref{eq:ell-N-limit} and the
bound~\eqref{eq:correction-vanishes} on $\mathrm{II}_N(y)$ yields
 \begin{align*}
&\sup_{x \in B_{(1-\delta)R_N}} \big|\rho_{N,R_N}(x)-\rho_\infty(\kappa)\big|
=
\sup_{y\in B_{1-\delta}}\big|\rho_{N,R_N}(R_Ny)-\rho_\infty(\kappa)\big|\\
&\ \le\
\Big|\frac{1}{4\pi}\ln(1+\ell_N)-\rho_\infty(\kappa)\Big|
+\sup_{y\in B_{1-\delta}}|I_{N,2}(y)|
+\sup_{y\in B_{1-\delta}}\mathrm{II}_N(y) \to 0 \quad \text{as } N \to \infty,
\end{align*}
which implies \eqref{eq:bulk-limit}. The proof is complete.
\end{proof}

\subsection{Energy-per-state limit}
 \label{subsec:fixed-energy}
Define
 \begin{equation} \label{compute:H}
 H(\kappa)
:=\frac{4\pi^2\bigl(4-\kappa\ln(1+4/\kappa)\bigr)}{\ln(1+4/\kappa)}, \quad \kappa>0.
\end{equation}
 Then with $\rho_\infty$ as in \eqref{eq:rho-infty-def}, we have
 \begin{equation} \label{def:H}
H(\kappa)=\pi \kappa \,\Theta_{4\pi}(\rho_\infty(\kappa)), \quad \kappa>0.
\end{equation}

 \begin{lemma}[Lower bound]
\label{lem:jensen-lower}
We have
\begin{equation} \label{eq:jensen-lower-bound}
\liminf_{N\to\infty}\frac{1}N\int_{\R^2}\Theta_{4\pi}(\rho_{N,R_N})\,dx\ \geq H(\kappa).
\end{equation}
\end{lemma}
 \begin{proof}
Note that $\Theta_{4\pi}$ is nonnegative, increasing and convex on $[0,\infty)$. Since $\rho_{N,R_N}\ge0$ vanishes outside $B_{R_N}$ and $\Theta_{4\pi}(0)=0$, by Jensen's
inequality for the convex function $\Theta_{4\pi}$
on the probability space $(B_{R_N},dx/|B_{R_N}|)$, we derive
\begin{equation} \label{lim-energy}
\frac{1}{|B_{R_N}|}\int_{\R^2}\Theta_{4\pi}(\rho_{N,R_N})\,dx
=
\frac{1}{|B_{R_N}|}\int_{B_{R_N}}\Theta_{4\pi}(\rho_{N,R_N})\,dx
\ \ge\
\Theta_{4\pi}\Big(\frac{1}{|B_{R_N}|}\int_{B_{R_N}}\rho_{N,R_N}\,dx\Big).
\end{equation}
By Lemma~\ref{lem:density-family} and the identity $|B_{R_N}|=\pi R_N^2=\pi \kappa_N N$, we obtain
\begin{equation} \label{lim-mass}
\frac{1}{|B_{R_N}|}\int_{B_{R_N}}\rho_{N,R_N}\,dx
=
\frac{1}{\pi \kappa_N N}\int_{\R^2}\rho_{N,R_N}\,dx
\to
\frac{g(\kappa)}{\pi \kappa}=\rho_\infty(\kappa)
\quad \text{as } N \to \infty.
\end{equation}
Multiplying \eqref{lim-energy} by $|B_{R_N}|/N = \pi \kappa_N$, then using \eqref{lim-mass}, $\kappa_N \to \kappa$ and the continuity of $\Theta_{4\pi}$, we obtain
$$ \liminf_{N\to\infty}\frac{1}N\int_{\R^2}\Theta_{4\pi}(\rho_{N,R_N})\,dx\ \geq \pi \kappa \Theta_{4\pi}(\rho_\infty(\kappa)) = H(\kappa),
$$
which is ~\eqref{eq:jensen-lower-bound}.
\end{proof}

 \begin{lemma}[Energy limit]
\label{lem:exp-density-limit}
We have
 \begin{equation} \label{eq:exp-density-limit}
\lim_{N\to\infty}\frac{1}N\int_{\R^2}\Theta_{4\pi}(\rho_{N,R_N}(x))\,dx
=
H(\kappa).
\end{equation}
\end{lemma}

\begin{proof}
Fix $\delta\in(0,1)$. By Lemma~\ref{lem:pointwise-bulk} and the
continuity of $\Theta_{4\pi}$, we have
\[
 \sup_{x\in B_{(1-\delta)R_N}}
 \big|\Theta_{4\pi}(\rho_{N,R_N}(x))-
       \Theta_{4\pi}(\rho_\infty(\kappa))\big|
 \to0\qquad\text{as }N\to\infty.
\]
This, together with  $|B_{(1-\delta)R_N}|=\pi(1-\delta)^2R_N^2=\pi(1-\delta)^2 \kappa_N N$ and the boundedness of $\{\kappa_N\}_{N \geq 1}$, implies
\[
\int_{B_{(1-\delta)R_N}}\Theta_{4\pi}(\rho_{N,R_N})\,dx
=
\pi(1-\delta)^2\kappa_N N\,\Theta_{4\pi}(\rho_\infty(\kappa))+o(N)_{N\to\infty}.
\]
Next, by using the fact that $\Theta_{4\pi}$ is increasing, the bound  $\rho_{N,R_N}\le \tilde C$ everywhere in $\R^2$ (Lemma~\ref{lem:uniform-rhoN}) and 
 $|B_{R_N}\setminus B_{(1-\delta)R_N}|=\pi
R_N^2(2\delta-\delta^2)\le2\delta\pi\kappa_N N$, we derive
 \[
\int_{B_{R_N}\setminus B_{(1-\delta)R_N}}\Theta_{4\pi}(\rho_{N,R_N})\,dx
\ \le\
2\delta\pi \kappa_N N\,\Theta_{4\pi}(\tilde C).
\]
From the above two estimates and $\kappa_N \to \kappa$, we obtain
 \[
\limsup_{N\to\infty}\frac{1}N\int_{\R^2}\Theta_{4\pi}(\rho_{N,R_N})\,dx
\ \le\
\pi(1-\delta)^2\kappa\,\Theta_{4\pi}(\rho_\infty(\kappa))+2\delta\pi \kappa\,\Theta_{4\pi}(\tilde C).
\]
Letting $\delta \to 0$ gives
 \[
\limsup_{N \to \infty} \frac{1}{N}\int_{\R^2} \Theta_{4\pi}(\rho_{N,R_N})\,dx\le\pi \kappa \,\Theta_{4\pi}(\rho_\infty(\kappa))=H(\kappa).
\]
This, combined with the lower bound \eqref{eq:jensen-lower-bound}, implies \eqref{eq:exp-density-limit}. The proof is complete.
\end{proof}

Recall that $\beta:=\ln(1+4/\kappa)$. Since
$\rho_\infty(\kappa)=\beta/(4\pi)$ and
$e^\beta-1=4/\kappa$, from the definition of $\Theta_{4\pi}$, we derive 
\begin{equation} \label{rela:Theta}
 \Theta_{4\pi}(\rho_\infty(\kappa))
 =\frac{e^\beta-1-\beta}{\beta/(4\pi)}
 =\frac{4\pi(4/\kappa-\beta)}{\beta}.
\end{equation}

\begin{proof}[\textbf{Proof of Theorem~\ref{thm:intro-intermediate-density}}]
The limit \eqref{lim:Dirichlet} is proved in Lemma~\ref{lem:density-family}. Finally, the limit \eqref{est:coupling} follows from Lemma~\ref{lem:exp-density-limit} and \eqref{compute:H}. The proof is complete.
\end{proof}

 \subsection{Asymptotic optimality of the localized mass}
\label{subsec:variational-optimality}

We now show that the mass limit of the Dirichlet ball family  \eqref{u_j^(R)} is the asymptotic variational supremum for the localized  mass.

Let $\cA=-\Delta+1$ on $L^2(\R^2)$ and define
\[
 \mathcal S_R:L^2(\R^2)\to L^2(B_R),
 \qquad \mathcal S_Rv:=\big(\cA^{-1/2}v\big)|_{B_R},
\]
as well as the restricted resolvent
\[
 \cQ_R:=\mathcal S_R\mathcal S_R^*=\1_{B_R}\cA^{-1}\1_{B_R}
 \quad\text{on }L^2(B_R)
\]
where $\1_{B_R}$ denotes the characteristic function on $B_R$. Here and below, functions on $B_R$ are identified with their zero
extensions when acted on by a whole-space operator.

\begin{lemma}[Ky Fan reduction]
\label{lem:variational-ky-fan}
Let $N\in\N$ and $R>0$, and let $\mathfrak{m}_N(R)$ be defined by \eqref{eq:intro-DN-def}. Then the operator $\cQ_R$ is positive and compact, and
\begin{equation} \label{KyFan}
 \mathfrak{m}_N(R)=\sum_{k=1}^N \nu_k(\cQ_R),
\end{equation}
where $\nu_1(\cQ_R) \geq \nu_2(\cQ_R) \geq \cdots$ denote its eigenvalues, counted with multiplicity.
\end{lemma}
\begin{proof}
\noindent\textbf{Step 1:} We will prove compactness of $\mathcal S_R$.
Indeed, we note that the operator $\cA^{-1/2}$ maps $L^2(\R^2)$ continuously into $H^1(\R^2)$, and restriction maps $H^1(\R^2)$ continuously into
$H^1(B_R)$. Moreover, the embedding of $H^1(B_R)$ into $L^2(B_R)$ is compact by the Rellich--Kondrachov theorem. Therefore, their composition
$\mathcal S_R$ is compact.  Both operators $\mathcal S_R^*\mathcal S_R$ and $\cQ_R=\mathcal S_R\mathcal S_R^*$ are positive and compact, and they have the same nonzero eigenvalues with the same multiplicities. \medskip

\noindent\textbf{Step 2:} We will prove that the two variational problems are equivalent.  For any $H^1$-orthonormal family $\{u_n\}_{n=1}^N \subset H^1(\R^2)$, put $v_n := \cA^{1/2} u_n$. Then $\mathcal S_Rv_n = u_n |_{B_R}$, $\{v_n\}_{n=1}^N$ forms an $L^2(\R^2)$-orthonormal family and 
\[ \int_{B_R} \sum_{n=1}^N |u_n|^2 \, dx = \sum_{n=1}^N \| \mathcal S_R v_n \|_{L^2(B_R)}^2 = \sum_{n=1}^N \langle \mathcal S_R^* \mathcal S_R v_n, v_n \rangle_{L^2(\R^2)}. \]

Conversely, because $\cA^{1/2}$ acts as a bijection between $H^1(\R^2)$ and $L^2(\R^2)$, every $L^2$-orthonormal family $\{v_n\}_{n=1}^N$ uniquely corresponds to an $H^1$-orthonormal family $\{u_n\}_{n=1}^N$ via the inverse mapping $u_n = \cA^{-1/2} v_n$. Consequently, optimizing over all possible $H^1$-orthonormal families $\{u_n\}_{n=1}^N$ is the same as optimizing over all $L^2$-orthonormal families $\{v_n\}_{n=1}^N$. This and  Step 1 imply
\begin{equation} \label{eq:variational-supremum}  \mathfrak{m}_N(R) = \sup_{\substack{\{v_n\}_{n=1}^N \subset L^2(\R^2) \\ L^2\text{-orthonormal}}} \sum_{n=1}^N \langle \mathcal S_R^* \mathcal S_R v_n, v_n \rangle_{L^2(\R^2)}.
\end{equation}

\noindent\textbf{Step 3:} We will prove Ky Fan's principle.
 Let $\cT\ge0$ be compact and self-adjoint on $L^2(\R^2)$.
Choose an orthonormal eigenbasis $\{e_j\}_{j\in I }$ of
$(\ker\cT)^\perp$ corresponding to its positive eigenvalues
$\nu_j$, listed in non-increasing order. Here $I=\N$ if $\cT$ has
infinite rank, and $I=\{1,\ldots,j_*\}$ if its rank is $j_*<\infty$;
when $j_*=0$, $I=\emptyset$. In the finite-rank case set $\nu_j=0$ for $j>j_*$.

For any $L^2$-orthonormal family $\{v_n\}_{n=1}^N$, set $a_j:=\sum_{n=1}^N|\langle v_n,e_j\rangle|^2$ for $j \in I$.
By Bessel's inequality, $0\le a_j\le1$ and
$
 \sum_{j\in I} a_j
 =\sum_{n=1}^N\|P_{(\ker\cT)^\perp}v_n\|_{L^2}^2\le N.
$
On the other hand, by the spectral theorem, we get
$\sum_{n=1}^N\langle\cT v_n,v_n\rangle
 =\sum_{j\in I}\nu_j a_j$.

If $\cT$ has at least $N$ positive eigenvalues, then thanks to the monotonicity, we obtain
\begin{align*}
 \sum_{j\in I}\nu_j a_j
 &\leq \sum_{j=1}^N\nu_j a_j
       +\nu_N\sum_{\substack{j\in I \\j>N}}a_j \leq \sum_{j=1}^N\bigl[\nu_j a_j+\nu_N(1-a_j)\bigr]
 \leq \sum_{j=1}^N\nu_j.
\end{align*}
Equality is attained by $v_n=e_n$ for $1\le n\le N$.
If $j_*<N$, then $\sum_{j=1}^{j_*}\nu_j a_j\leq\sum_{j=1}^{j_*}\nu_j$ and 
equality is attained by the $j_*$ positive-eigenvalue vectors together
with $N-j_*$ orthonormal vectors from $\ker\cT$. 
Thus, in all cases, we obtain 
\begin{equation}\label{sup-leq-sum}
 \sup_{\substack{\{v_n\}_{n=1}^N\subset L^2(\R^2)\\
                  L^2\text{-orthonormal}}}
 \sum_{n=1}^N\langle\cT v_n,v_n\rangle
 =\sum_{j=1}^N\nu_j(\cT).
\end{equation}

Applying~\eqref{sup-leq-sum} to $\cT=\mathcal S_R^*\mathcal S_R$
and using~\eqref{eq:variational-supremum} and Step 1 yields \eqref{KyFan}. 
\end{proof}

Recall that the eigenvalues $\lambda_k(R)$ of $\cA_R=-\Delta+1$ are given in~\eqref{eq:lambdaj}.

\begin{lemma}[Lower bound]
\label{lem:variational-lower}
For every $R>0$ and $k\ge1$,
$\nu_k(\cQ_R) \ge  \lambda_k(R)^{-1}$.
Consequently, for any $N \in \N$ and $R>0$, with $\mathfrak{m}_N(R)$ as in \eqref{eq:intro-DN-def} and  $\rho_{N,R}$ as in  \eqref{rho_NR}, we have
\begin{equation}
 \mathfrak{m}_N(R)\ge\sum_{k=1}^N \lambda_k(R)^{-1}
 =\int_{\R^2}\rho_{N,R}(x)\,dx.
 \label{eq:variational-lower}
\end{equation}
\end{lemma}

\begin{proof}
For $v\in L^2(B_R)$, the quadratic-form duality formula gives
\begin{equation} \label{eq:Q_R-quadratic-form}
 \langle \cQ_R v,v\rangle
 =\sup_{0\ne w\in H^1(\R^2)}
 \frac{|\langle v,w\rangle_{L^2(B_R)}|^2}{\int_{\R^2}(|\nabla w|^2+|w|^2)\,dx}.
\end{equation}
Restricting the supremum to the zero extensions of functions in
$H_0^1(B_R)$ leads to
\[
 \langle \cQ_R v,v\rangle
 \geq \sup_{0\ne w\in H_0^1(B_R)}
  \frac{|\langle v,w\rangle_{L^2(B_R)}|^2}{\int_{B_R}(|\nabla w|^2+|w|^2)\,dx}
 =\langle \cA_R^{-1}v,v\rangle_{L^2(B_R)}.
\]
Therefore, by using the min-max principle~\cite[Section~XIII.1]{ReedSimonIV}, we derive 
$\nu_k(\cQ_R) \geq \lambda_k(R)^{-1}$. Finally, summing over $1 \leq k \leq N$ and using
Lemma~\ref{lem:variational-ky-fan} together with
\eqref{eq:rhoN-dirichlet} yield ~\eqref{eq:variational-lower}.
\end{proof}

The IMS argument below changes the radius by a factor tending to one.

\begin{lemma}[Upper bound]
\label{lem:variational-ims-upper}
There is a universal constant $C_2>0$ with the following property. For $R>0$, $\delta\in(0,1)$, and $R'=R(1+\delta)$,
\begin{equation}  \label{eq:variational-ims-eigenvalue}
 \nu_k(\cQ_R) \leq
 \left(1+\frac{C_2}{(\delta R)^2}\right) \lambda_k(R')^{-1}
 \quad \forall k \geq 1.
\end{equation}
Consequently, for any $N \in \N$, with $ \mathfrak{m}_N(R)$ as in \eqref{eq:intro-DN-def}, we have
\begin{equation} \label{eq:variational-ims-sum}
 \mathfrak{m}_N(R)\le
 \left(1+\frac{C_2}{(\delta R)^2}\right)
 \sum_{k=1}^N \lambda_k(R')^{-1}.
\end{equation}
\end{lemma}

\begin{proof}
Let $R>0$, $\delta \in (0,1)$, and let $\chi,\eta\in C^\infty(\R^2;[0,1])$ be radial functions such that $\chi^2+\eta^2=1$ in $\R^2$, $\chi=1$ on $B_R$,
$\operatorname{supp}\chi\Subset B_{R'}$, and
$|\nabla\chi|^2+|\nabla\eta|^2\le \frac{C_2}{(\delta R)^2}$. For every $w \in H^1(\R^2)$, by the IMS identity (see \cite{LieYau88}),
\[
 \|\chi w\|_{H^1(\R^2)}^2 + \|\eta w\|_{H^1(\R^2)}^2
 = \| w \|_{H^1(\R^2)}^2 +\int_{\R^2}(|\nabla\chi|^2+|\nabla\eta|^2)|w|^2\,dx,
\]
which implies 
\begin{equation}  \label{eq:variational-ims-form}
 \|\chi w\|_{H^1(\R^2)}^2 \leq 
 \left(1+\frac{C_2}{(\delta R)^2}\right) \| w \|_{H^1(\R^2)}^2.
\end{equation}

Let $\cE: L^2(B_R)\to L^2(B_{R'})$ be the operator extending functions by zero in $B_{R'} \setminus B_R$. Since
$\chi=1$ on $B_R$, for $v\in L^2(B_R)$ and $w \in H^1(\R^2)$, we have
$\langle v,w\rangle_{L^2(B_R)}=\langle \cE v,\chi w\rangle_{L^2(B_{R'})}$
and $\chi w\in H_0^1(B_{R'})$. From \eqref{eq:Q_R-quadratic-form} and \eqref{eq:variational-ims-form}, we derive
\begin{align*}
 \langle \cQ_R v,v\rangle_{L^2(B_R)}
& \leq \left(1+\frac{C_2}{(\delta R)^2}\right)
 \sup_{0 \neq \phi \in H_0^1(B_{R'})}
 \frac{|\langle \cE v, \phi \rangle_{L^2(B_{R'})}|^2}{\int_{B_{R'}}(|\nabla \phi |^2 + |\phi|^2)\,dx} \\ & =\left(1+\frac{C_2}{(\delta R)^2}\right)
 \langle \cA_{R'}^{-1}\cE v,\cE v\rangle_{L^2(B_{R'})}.
\end{align*}
Thus, as quadratic forms on $L^2(B_R)$,
$\cQ_R \leq \left(1+\frac{C_2}{(\delta R)^2}\right)
 \cE^*\cA_{R'}^{-1}\cE$.
Hence by the min--max principle~\cite[Section~XIII.1]{ReedSimonIV},
$\nu_k(\cE^*\cA_{R'}^{-1}\cE)\le\nu_k(\cA_{R'}^{-1})=\lambda_k(R')^{-1}$. Combining this with the preceding quadratic-form inequality and applying the min--max principle once more, we obtain
\[
 \nu_k(\cQ_R) \leq \left(1+\frac{C_2}{(\delta R)^2}\right)
 \nu_k(\cE^*\cA_{R'}^{-1}\cE)
 \leq \left(1+\frac{C_2}{(\delta R)^2}\right)
 \lambda_k(R')^{-1},
\]
which implies ~\eqref{eq:variational-ims-eigenvalue}.  Summing over $1 \leq k \leq N$  and using
Lemma~\ref{lem:variational-ky-fan}, we deduce
\eqref{eq:variational-ims-sum}.
\end{proof}

 \begin{proof}[\textbf{Proof of Theorem~\ref{thm:intro-variational-optimality}}]
From Lemma~\ref{lem:variational-lower} and
Lemma~\ref{lem:density-family}, we have
\begin{equation} \label{est:lower-DN}
 \liminf_{N\to\infty}\frac{\mathfrak{m}_N(R_N)}N\ge g(\kappa).
\end{equation}
For the upper bound, note that $R_N\to\infty$ as $N\to\infty$.
Define the sequences $\{\delta_N\}_{N\ge1}$ and $\{R_N'\}_{N\ge1}$ by
$\delta_N=R_N^{-1/2}$ and $R_N'=R_N(1+\delta_N)$.
Then $\delta_N\in(0,1)$ for all sufficiently large $N$. Moreover,
$\delta_N \to 0$, $\frac{C_2}{(\delta_NR_N)^2}=\frac{C_2}{R_N} \to 0$ and $\frac{(R_N')^2}{N} \to \kappa$.
By Lemma~\ref{lem:variational-ims-upper} and
Lemma~\ref{lem:density-family}, we derive
\begin{equation} \label{est:upper-DN}
 \limsup_{N\to\infty}\frac{\mathfrak{m}_N(R_N)}N
 \le\lim_{N\to\infty}
 \left(1+\frac{C_2}{R_N}\right)
 \frac{1}N\sum_{k=1}^N \lambda_k(R_N')^{-1}
 =g(\kappa).
\end{equation}
Combining \eqref{est:lower-DN} and \eqref{est:upper-DN} yields~\eqref{eq:variational-optimality}.
\end{proof}
\section{Low-density regime}
\label{sec:low-density}

Let $\{\kappa_N\}_{N\ge1}$ be a sequence of positive numbers such that $\kappa_N\to\infty$. Set $R_N=\sqrt{\kappa_N N}$, so $R_N\to\infty$. Recall that $\mu_j$ and $\psi_j$ denote the eigenvalues and corresponding eigenfunctions of the Dirichlet Laplacian $-\Delta$ in the unit ball $B_1$. Moreover, $\{\psi_j\}_{j\ge1}\subset H_0^1(B_1)$ forms an orthonormal basis of $L^2(B_1)$.
Let $\{u_j^{(R_N)}\}_{j=1}^N$ and $\rho_{N,R_N}$ be defined by
\eqref{u_j^(R)} and~\eqref{rho_NR}, respectively.
Throughout this section, write $\rho_N:=\rho_{N,R_N}$. In this regime, the relevant normalization is $\kappa_N\rho_N$. We will first establish uniform estimates on the
fixed unit ball and then perform a suitable rescaling to $B_{R_N}$. This allows us to prove
Theorem~\ref{thm:low-density}.

 \begin{proof}[\textbf{Proof of Theorem~\ref{thm:low-density}}]
From \eqref{eq:lambdaj} and noting that $R_N^2=\kappa_N N$, we have 
 $\lambda_j(R_N) = 1 + \frac{\mu_j}{\kappa_N N}$. Since $0 < \mu_1 \le \mu_2 \le \dots \le \mu_N$, it follows that  $\lambda_j(R_N)^{-1}$ satisfies
 \[
\left( 1 + \frac{\mu_N}{\kappa_N N} \right)^{-1} \le \lambda_j(R_N)^{-1} \le 1, \qquad \text{for all } 1 \le j \le N.
\]
By Weyl's law ~\eqref{eq:weyl-ball},  $\mu_N = 4N(1+o(1)_{N \to \infty})$, which implies that  $\frac{\mu_N}{\kappa_N N} = O(\kappa_N^{-1}) \to 0$ as $N \to \infty$. Using the mass identity~\eqref{eq:rhoN-dirichlet}, we have $\int_{\R^2} \rho_N \, dx = \sum_{j=1}^N \lambda_j(R_N)^{-1}$, hence 
 \begin{equation} \label{eq:dilute-mass-squeeze}
\left( 1 + \frac{\mu_N}{\kappa_N N} \right)^{-1} \le \frac{1}{N} \int_{\mathbb{R}^2} \rho_N \, dx \le 1, \quad 0 \leq 1 - \frac{1}{N} \int_{\mathbb{R}^2} \rho_N \, dx \leq \frac{\mu_N}{\kappa_N N} =O(\kappa_N^{-1}),
\end{equation}
which implies ~\eqref{eq:dilute-mass} with an error
$O(\kappa_N^{-1})$.

Define 
 \[
q_N(y) := \kappa_N \rho_N(R_N y), \quad y \in \overline{B_1}.
\] 
Using the explicit representation of $\rho_N$ in ~\eqref{rho_NR}, we deduce
  \begin{equation} \label{q_N}
 q_N(y) = \kappa_N \sum_{j=1}^N \lambda_j(R_N)^{-1} R_N^{-2} |\psi_j(y)|^2 = \frac{1}{N} \sum_{j=1}^N \lambda_j(R_N)^{-1} |\psi_j(y)|^2.
\end{equation}
By the definition of $S_{0,N}$ and keeping in mind that $\lambda_j(R_N)^{-1} \le 1$, we obtain $q_N(y) \le \frac{S_{0,N}(y)}{N}$. This, together with \eqref{eq:dilute-projector-global} (which is independent of $\kappa_N$ and thus also holds in the low-density regime) and the fact that $\rho_N$ vanishes outside $B_{R_N}$, leads to
\eqref{eq:dilute-global-bound}.

Next, by the definition of $S_{0,N}$, $q_N$, the inequality $\lambda_j(R_N) \leq \lambda_N(R_N)$ for any $j \leq N$, the relation $\lambda_j(R_N) = 1 + \frac{\mu_j}{\kappa_N N}$ and \eqref{eq:dilute-projector-global}, we derive
 \[
0 \le \frac{S_{0,N}(y)}{N} - q_N(y) = \frac{1}{N} \sum_{j=1}^N \left( 1 - \lambda_j(R_N)^{-1} \right) |\psi_j(y)|^2 \le \frac{\mu_N}{\kappa_N N} \frac{S_{0,N}(y)}{N} \le C_1 \frac{\mu_N}{\kappa_N N}.
\]
 By the triangle inequality and~\eqref{est:S/N-pi} (which is independent of $\kappa_N$ and thus also holds in the low-density regime), for every fixed $\delta\in(0,1)$ and every $y\in\overline{B_{1-\delta}}$, 
\begin{equation} \label{q_N:bulk}
\left| q_N(y) - \frac{1}{\pi} \right| \le \left| q_N(y) - \frac{S_{0,N}(y)}{N} \right| + \left| \frac{S_{0,N}(y)}{N} - \frac{1}{\pi} \right| \le C_1 \frac{\mu_N}{\kappa_N N} + \varepsilon_{N,\delta},
\end{equation}
which, together with the convergences $\lim_{N \to \infty} \frac{\mu_N}{\kappa_N N} = \lim_{N \to \infty} \frac{4}{\kappa_N} = 0$ and $\varepsilon_{N,\delta} \to 0$ as $N \to \infty$,  implies 
\begin{equation}\label{eq:dilute-bulk}
 \lim_{N\to\infty}\sup_{x\in B_{(1-\delta)R_N}}
 \left|\kappa_N\rho_{N,R_N}(x)-\frac{1}{\pi}\right|=0.
\end{equation}

 Fix $1\le p<\infty$ and $\delta\in(0,1)$. By the change of variable $x=R_Ny$, the global bound~\eqref{eq:dilute-global-bound} and the interior estimate~\eqref{q_N:bulk}, we obtain
\begin{align*}
\frac{1}{|B_{R_N}|} \int_{B_{R_N}} \left| \kappa_N \rho_N(x) - \frac{1}{\pi} \right|^p \, dx &= \frac{1}{\pi} \int_{B_1} \left| q_N(y) - \frac{1}{\pi} \right|^p \, dy \\
&= \frac{1}{\pi} \int_{B_{1-\delta}} \left| q_N(y) - \frac{1}{\pi} \right|^p \, dy + \frac{1}{\pi} \int_{B_1 \setminus B_{1-\delta}} \left| q_N(y) - \frac{1}{\pi} \right|^p \, dy \\
& \le (1-\delta)^2 \left( \varepsilon_{N,\delta} + C_1 \frac{\mu_N}{\kappa_N N} \right)^p + (2\delta - \delta^2) \left( C_1 + \frac{1}{\pi} \right)^p.
\end{align*}
Letting $N \to \infty$ and $\delta \downarrow 0$ successively in the above estimate implies \eqref{eq:dilute-Lp}. 

We infer from ~\eqref{eq:Theta-def} that
$0\le\Theta_{4\pi}(s)-8\pi^2s\le \widehat C s^2$ for $0\le s\le1$.
By~\eqref{eq:dilute-global-bound}, $\|\rho_N\|_{L^\infty}
\le C/\kappa_N\to0$ as $N \to \infty$. Hence for $N$ large enough, $\rho_N(x) \leq 1$ for a.e. $x \in \R^2$. Therefore, for all sufficiently large $N$,
\begin{align*}
 0&\le\frac{1}N\int_{\R^2}\Theta_{4\pi}(\rho_N)\,dx
       -\frac{8\pi^2}{N}\int_{\R^2}\rho_N\,dx
 \le\frac{\widehat C}{N}\int_{\R^2}\rho_N^2\,dx
 \le \widehat C\|\rho_N\|_{L^\infty}
       \frac{1}N\int_{\R^2}\rho_N\,dx
 \le\frac{C_2}{\kappa_N},
\end{align*}
where $C_2$ is independent of $N$. This, together with~\eqref{eq:dilute-mass}, implies
\eqref{eq:dilute-energy}. 
\end{proof}
\begin{remark} \label{remark:dilute}

(i) In view of the proof of \eqref{eq:dilute-energy}, we see that
$$
 \frac{1}N\int_{\R^2}\Theta_{4\pi}(\rho_{N})\,dx
 =8\pi^2+O(\kappa_N^{-1})_{N \to \infty}.
$$

(ii) The convergence in \eqref{eq:dilute-Lp} holds for any finite $1 \le p < \infty$, but not for $p=\infty$ since 
\[
 \left\|\kappa_N\rho_N-\frac1\pi\right\|_{L^\infty(B_{R_N})}
 \ge\frac1\pi, \quad \forall N \geq 1,
\]
due to the continuity of $\rho_N$ on $\overline{B_{R_N}}$ and the fact that $\rho_N=0$ on $\partial B_{R_N}$.
\end{remark}

\section{High-density regime}
\label{sec:high-density}

In this section, let $\{\kappa_N\}_{N\ge1}$ be a sequence
of positive numbers such that $\kappa_N\to0$. Set $R_N=\sqrt{\kappa_N N}$
and assume that $R_N\to\infty$.
Let $\{u_j^{(R_N)}\}_{j=1}^N$ and $\rho_{N,R_N}$ be defined by
\eqref{u_j^(R)} and~\eqref{rho_NR}, respectively. Throughout this section,  we write
$u_j:=u_j^{(R_N)}$ and $\rho_N:=\rho_{N,R_N}$.

Define the sequence $\{\beta_N\}_{N\ge1}$ by
\[
 \beta_N=\ln\Big(1+\frac{4}{\kappa_N}\Big).
\]
Then $\beta_N\to\infty$. 
 The profile in \eqref{eq:rho-infty-def} suggests 
$\rho_\infty(\kappa_N)=\beta_N/(4\pi)$. We therefore study
$\rho_N/\beta_N$. 

We first establish an upper bound on $\rho_N/\beta_N$ which relies only
on the monotonicity $\lambda_j(R_N)\le\lambda_N(R_N)$ for $j\le N$ and on the optimized scalar estimate in Lemma~\ref{lem:optimized-Phi}.

\begin{lemma}[Uniform high-density upper bound]
\label{lem:uniform-upper}
The following estimate holds
\begin{equation} \label{eq:uniform-upper-claim}
\limsup_{N\to\infty} \left(\frac{1}{\beta_N} \sup_{x\in\R^2}\ \rho_N(x)\right)  \leq \frac{1}{4\pi}.
\end{equation}
\end{lemma}

\begin{proof}
Let $t>0$. In view of the proof of Lemma~\ref{lem:uniform-rhoN},
since $\lambda_j(R_N)$ is non-decreasing in $j$,
$e^{-2t\lambda_j(R_N)}\ge e^{-2t\lambda_N(R_N)}$ for $j\le N$, hence
\[
e^{-2t\lambda_N(R_N)}\rho_N(x)
\ \le\
\sum_{j=1}^N e^{-2t\lambda_j(R_N)}|u_j(x)|^2
\ \le\
\frac{1}{4\pi}\ln\Big(1+\frac{1}{2t}\Big)
\quad\text{for a.e. }x\in\R^2.
\]
 For each fixed $t>0$, continuity of $\rho_N$ extends its
a.e. bound by a constant to every $x\in\R^2$.
The bound therefore holds for every pair $(x,t)$, so taking the
infimum over $t>0$ gives, for every $N\ge1$,
\begin{equation} \label{eq:rho-Phi-bound}
\rho_N(x)  \leq \frac{1}{4\pi}\,\Phi(\lambda_N(R_N))
\quad\text{for every }x\in\R^2,
\end{equation}
with $\Phi$ defined in \eqref{def:Phi}. By Weyl's
law~\eqref{eq:weyl-ball} and the relation $R_N^2=\kappa_N N$,
 \[
\lambda_N(R_N)=1+\frac{\mu_N}{R_N^2}
= 1+\frac{4}{\kappa_N}(1+\eta_N)
 \to \infty \quad \text{as } N \to \infty.
\]
Here the sequence $\{\eta_N\}_{N\ge1}$ is defined by
$\mu_N=4N(1+\eta_N)$ and satisfies $\eta_N \to 0$ as $N \to \infty$.
Hence
\[
\ln(\lambda_N(R_N)) - \beta_N
=
\ln\Big(1+\frac{4}{\kappa_N}(1+\eta_N)\Big)-\ln\Big(1+\frac{4}{\kappa_N}\Big)
=
\ln\Big(1+\frac{4\eta_N}{\kappa_N+4}\Big)
\to 0 \quad \text{as } N \to \infty,
\]
which implies $\ln\lambda_N(R_N)=\beta_N+o(1)_{N \to \infty}$. Since
$\beta_N\to\infty$, we have
$\ln \lambda_N(R_N)/\beta_N\to1$. From Lemma~\ref{lem:optimized-Phi}, we have
\[
 \frac{\ln\ln \lambda_N(R_N)}{\beta_N}\to  0
 \quad \text{ and } \quad
 \Phi(\lambda_N(R_N))\le\beta_N(1+o(1)_{N \to \infty}).
\]
Together with
\eqref{eq:rho-Phi-bound}, this yields
$\operatorname*{ess\,sup}_{x\in\R^2}\rho_N(x)
\ \le\ \frac{\beta_N}{4\pi}\,(1+o(1)_{N \to \infty})$.
Since $\rho_N$ is continuous by Lemma~\ref{lem:dirichlet-family}, its
supremum equals its essential supremum.
This proves~\eqref{eq:uniform-upper-claim}.
\end{proof}
\begin{remark}
The assumption $R_N \to \infty$ is not needed in the proof of Lemma \ref{lem:uniform-upper}.
\end{remark}

We next establish the matching limit.

\begin{lemma}[Averaged mass limit]
 \label{lem:normalized-mass-limit}
The following limit holds
 \begin{equation}
\lim_{N \to \infty}\frac{1}{\beta_N |B_{R_N}|}\int_{\R^2}\rho_N(x)\,dx = \frac{1}{4\pi}.
\label{eq:normalized-mass-highdensity}
\end{equation}
\end{lemma}
\begin{proof}
We first derive an upper bound. Since $\rho_N\ge0$ vanishes outside $B_{R_N}$, we have  $\frac{1}{|B_{R_N}|}\int_{\R^2}\rho_N\,dx \leq \sup_{x\in\R^2}\rho_N(x)$. Hence by Lemma~\ref{lem:uniform-upper}, we obtain
\begin{equation} \label{limsup:rho_N}
\limsup_{N \to \infty} \frac{1}{\beta_N|B_{R_N}|}\int_{\R^2}\rho_N\,dx \leq \frac{1}{4\pi}.
\end{equation}

Next, we show a lower bound. As in the proof of Lemma~\ref{lem:density-family},
we write
 \[
 \frac{1}{N}\int_{\R^2}\rho_N\,dx
 =\frac{1}N\sum_{j=1}^N h_{N,j}
 \quad \text{where} \quad
 h_{N,j}:=\frac{\kappa_N N}{\mu_j+\kappa_N N}.
\]
Fix
$\varepsilon\in(0,4)$. By Weyl's law~\eqref{eq:weyl-ball} there is
$J_0=J_0(\varepsilon)$ with  $\mu_j \leq (4+\varepsilon)j$
for all $j\ge J_0$. For each $N$, set
\[
 \phi_{\varepsilon,N}(t)
 :=\frac{\kappa_N}{(4+\varepsilon)t+\kappa_N},
 \quad 0\le t\le1.
\]
Then $h_{N,j}\ge\phi_{\varepsilon,N}(j/N)$ whenever $j\ge J_0$.
Consequently, for $N \geq J_0$, we have
\[
 \frac{1}N\sum_{j=1}^N h_{N,j}
 \ge \frac{1}N\sum_{j=J_0}^N\phi_{\varepsilon,N}(j/N).
\]
Since $\phi_{\varepsilon,N}$ is decreasing on $[0,1]$ with values in $[0,1]$,
the standard monotone Riemann-sum estimate  gives
\[
\frac{1}N\sum_{j=1}^N\phi_{\varepsilon,N}(j/N)
\ \ge\
\int_0^1\phi_{\varepsilon,N}(x)\,dx-\frac{1}N
=
g_\varepsilon(\kappa_N)-\frac{1}N,
\quad \text{where }
g_\varepsilon(\kappa):=\frac{\kappa}{4+\varepsilon}\ln\Big(1+\frac{4+\varepsilon}{\kappa}\Big).
\]
The two preceding estimates and the fact that
$\phi_{\varepsilon,N}\le1$ imply
$\frac{1}N\sum_{j=1}^N h_{N,j}
\ \ge\ g_\varepsilon(\kappa_N)-\frac{J_0}N$.
Hence
\begin{equation} \label{est:lower-mass}
\frac{1}{Ng(\kappa_N)}\int_{\R^2}\rho_N\,dx
\ \ge\
\frac{g_\varepsilon(\kappa_N)}{g(\kappa_N)}
-\frac{J_0}{Ng(\kappa_N)}.
\end{equation}
Note that $\kappa_N\to0$ and
\[
 Ng(\kappa_N)
 =\frac{R_N^2}{4}\ln\Big(1+\frac{4}{\kappa_N}\Big)
 =\frac{R_N^2\beta_N}{4}\to  \infty.
\]
Moreover, $\lim_{\kappa\downarrow0}\frac{g_\varepsilon(\kappa)}{g(\kappa)}
 =\frac{4}{4+\varepsilon}$.
Therefore, by letting $N\to\infty$ in \eqref{est:lower-mass}, we obtain 
\[
\liminf_{N\to\infty}\frac{1}{Ng(\kappa_N)}\int_{\R^2}\rho_N\,dx\ \ge\ \frac{4}{4+\varepsilon}.
\]
Letting $\varepsilon\to0$ and using
$\beta_N|B_{R_N}|=4\pi Ng(\kappa_N)$, we get
\[
\liminf_{N \to \infty} \frac{1}{\beta_N|B_{R_N}|}\int_{\R^2} \rho_N\,dx \geq \frac{1}{4\pi}.
\]
This and the upper bound~\eqref{limsup:rho_N} imply
 \eqref{eq:normalized-mass-highdensity}.
\end{proof}

Interestingly, in contrast to the intermediate- and low-density regimes, both  the mass and energy per state converge to zero as $N \to \infty$ in the high-density regime, as stated in the following result.

\begin{proposition} \label{prop:converge-zero}
The following limits hold
\begin{align}
 \lim_{N\to\infty}\frac1N\int_{\R^2}\rho_{N,R_N}(x)\,dx
 &=0,\label{eq:high-density-mass}\\
 \lim_{N\to\infty}\frac1N\int_{\R^2}
 \Theta_{4\pi}(\rho_{N,R_N}(x))\,dx
 &=0.\label{eq:high-density-energy}
\end{align}
\end{proposition}
\begin{proof}
We first prove~\eqref{eq:high-density-mass}.
By Lemma~\ref{lem:normalized-mass-limit} and
$|B_{R_N}|=\pi\kappa_N N$,
\[
 \frac1N\int_{\R^2}\rho_N(x)\,dx
 =\pi\kappa_N\beta_N
   \frac1{\beta_N|B_{R_N}|}\int_{\R^2}\rho_N(x)\,dx
 =\frac{\kappa_N\beta_N}{4}(1+o(1)_{N \to \infty}) \to 0,
\]
since $\kappa_N\ln(1+4/\kappa_N) \to 0$.
This implies~\eqref{eq:high-density-mass}.

To prove~\eqref{eq:high-density-energy}, we put
\[
 P_N(x):=\sum_{j=1}^N\lambda_j(R_N)|u_j(x)|^2.
\]
The functions $\{\lambda_j(R_N)^{1/2}u_j\}_{j=1}^N$ are $L^2(B_{R_N})$-orthonormal
Dirichlet eigenfunctions. This finite orthonormal set can be extended to a complete orthonormal eigenbasis of $L^2(B_{R_N})$, even if the selected functions span only a proper subspace of the eigenspace associated with the maximal eigenvalue $\lambda_N(R_N)$.
Since $\lambda_j(R_N) \le \lambda_N(R_N)$ for $j\le N$, the spectral
expansion and Lemma~\ref{lem:dirichlet-monotone} give, for a.e. $x\in B_{R_N}$,
\[
 P_N(x)
 \le e\sum_{j=1}^N e^{-\frac{\lambda_j(R_N)}{\lambda_N(R_N)}}
                    \lambda_j(R_N)|u_j(x)|^2
 \le e^{1-\frac{1}{\lambda_N(R_N)}}
       k_{-\Delta}^{B_{R_N}}(\lambda_N(R_N)^{-1},x,x)
 \le\frac{e\lambda_N(R_N)}{4\pi}.
\]
Using $1-e^{-s}\le s$ for $s\ge0$ and the Dirichlet diagonal
bound~\eqref{eq:Q-dirichlet}, we obtain
\begin{align*}
 \rho_N(x)
 &=\sum_{j=1}^N
       (1-e^{-\frac{\lambda_j(R_N)}{\lambda_N(R_N)}})|u_j(x)|^2
   +\sum_{j=1}^N e^{-\frac{\lambda_j(R_N)}{\lambda_N(R_N)} }|u_j(x)|^2\\
 &\le\frac{P_N(x)}{\lambda_N(R_N)}
       +Q_{1/(2\lambda_N(R_N))}^{R_N}(x)
 \le\frac{e+\ln(1+\lambda_N(R_N))}{4\pi}.
\end{align*}
By the continuity of $\rho_N$ and the fact that $\rho_N = 0$ in $B_{R_N}^c$, we derive that
$0 \le \rho_N(x)\le \frac{e+\ln(1+\lambda_N(R_N))}{4\pi}$ for every $x\in\R^2$.
Since $\Theta_{4\pi}$ is increasing on $[0,\infty)$, vanishes at zero,
and satisfies $\Theta_{4\pi}(s)\le e^{4\pi s}/s$ for $s>0$,
\begin{equation} \label{est:energy-per-state}
 0\le\frac1N\int_{\R^2}\Theta_{4\pi}(\rho_N(x))\,dx
 \le\pi\kappa_N\Theta_{4\pi}\left(\frac{e+\ln(1+\lambda_N(R_N))}{4\pi}\right)
 \le\frac{4\pi^2e^e\kappa_N(1+\lambda_N(R_N))}
          {e+\ln(1+\lambda_N(R_N))}.
\end{equation}
Finally, \eqref{eq:lambdaj} and Weyl's law~\eqref{eq:weyl-ball} imply
\[
 \lambda_N(R_N)=1+\frac{\mu_N}{\kappa_N N}\to\infty,
 \qquad
 \kappa_N(1+\lambda_N(R_N))=2\kappa_N+\frac{\mu_N}{N}\to4.
\]
Therefore, letting $N \to \infty$ in \eqref{est:energy-per-state} yields
\eqref{eq:high-density-energy}. The proof is complete.
\end{proof}
\begin{remark}
Set
\[
 \overline\rho_N:=\frac{1}{|B_{R_N}|}
 \int_{B_{R_N}}\rho_{N,R_N}(x)\,dx.
\]
From the normalized mass limit~\eqref{eq:normalized-mass-highdensity}, we see that $\overline\rho_N/\beta_N\to1/(4\pi)$. Since
$\beta_N=\ln(1+4/\kappa_N)\to\infty$, we have
$\overline\rho_N\to\infty$ as $N \to \infty$.
By Jensen's inequality, we derive
\[
\begin{aligned}
 \frac{\int_{\R^2}\Theta_{4\pi}(\rho_{N,R_N}(x))\,dx}
      {\int_{\R^2}\rho_{N,R_N}(x)\,dx}
 \ge\frac{\Theta_{4\pi}(\overline\rho_N)}{\overline\rho_N}
 =\frac{e^{4\pi\overline\rho_N}-1-4\pi\overline\rho_N}
       {\overline\rho_N^2}\to\infty \quad \text{as } N \to \infty.
\end{aligned}
\]
Thus in the high-density regime, estimate \eqref{eq:false-uniform-mass} fails, even though the mass and energy per state asymptotically vanish.
\end{remark}

With the normalized density $\rho_N/\beta_N$, we can show the convergence of the averaged mass and energy in the high-density regime.

\begin{proof}[\textbf{Proof of Theorem~\ref{thm:intro-high-density}}]

First we prove \eqref{eq:Lp-highdensity}. For  $h\in L^1(B_{R_N})$, define
$$\langle h\rangle_N:=|B_{R_N}|^{-1}\int_{B_{R_N}}h\,dx.
$$ 

 By Lemma~ \ref{lem:normalized-mass-limit}, $\langle \frac{\rho_N}{\beta_N} \rangle_N \to \frac{1}{4\pi}$ and by
Lemma~\ref{lem:uniform-upper}, 
\[
 \limsup_{N\to\infty}\sup_{x\in\R^2}
 \frac{\rho_N(x)}{\beta_N}\le\frac{1}{4\pi}.
\]
Using the positive-part identity and the preceding mass limit \eqref{eq:normalized-mass-highdensity}, we obtain
\begin{align*}
 \left\langle\left|\frac{\rho_N}{\beta_N}-\frac{1}{4\pi}\right|\right\rangle_N
 &=2\left\langle\left(\frac{\rho_N}{\beta_N}-\frac{1}{4\pi}\right)_+\right\rangle_N
   -\left(\left\langle\frac{\rho_N}{\beta_N}\right\rangle_N-\frac{1}{4\pi}\right)\\
 &\le 2\left(\sup_{x\in\R^2}\frac{\rho_N(x)}{\beta_N}-\frac{1}{4\pi}\right)_+
   +\left|\left\langle\frac{\rho_N}{\beta_N}\right\rangle_N-\frac{1}{4\pi}\right| \to 0 \qquad\text{as }N\to\infty,
\end{align*}
which proves~\eqref{eq:Lp-highdensity} for $p=1$.

For all sufficiently large $N$, we have $0 \leq \frac{\rho_N}{\beta_N} \leq 1 + \frac{1}{4\pi}$. Thus, if
$1<p<\infty$,
\[
 \Big \langle \Big |\frac{\rho_N}{\beta_N} - \frac{1}{4\pi} \Big |^p \Big \rangle_N
 \leq \Big(1 + \frac{1}{4\pi}\Big)^{p-1}\Big \langle \Big |\frac{\rho_N}{\beta_N} -  \frac{1}{4\pi} \Big | \Big \rangle_N \to 0 \quad \text{as } N \to \infty,
\]
which proves~\eqref{eq:Lp-highdensity} for $p>1$.

Next we prove \eqref{eq:sup-rho-limit}. By Lemma~ \ref{lem:normalized-mass-limit},
\[
\liminf_{N\to\infty} \left( \frac{1}{\beta_N}\sup_{x\in\R^2}\rho_N(x) \right)
\ \ge\
\lim_{N\to\infty} \left(\frac{1}{\beta_N|B_{R_N}|}\int_{\R^2}\rho_N\,dx \right) = \frac{1}{4\pi}.
\]
This and Lemma~\ref{lem:uniform-upper} yield \eqref{eq:sup-rho-limit}.

Finally, note that $\Theta_{4\pi}$ is Lipschitz on the common bounded range of $\rho_N/\beta_N$ for all sufficiently large $N$. Thus \eqref{eq:Lp-highdensity} with $p=1$ and $\Theta_{4\pi}(1/(4\pi))=4\pi(e-2)$ yield \eqref{eq:intro-Theta-convergence}. 
\end{proof}


\appendix
 \section{Scalar endpoint inequality and comparisons}
\label{app:scalar-exact-growth}

In this appendix, we prove inequality \eqref{eq:scalar-exact-growth} and show its equivalence to~\eqref{eq:IMN-exact-growth-intro} and its implication for Ruf's inequality \eqref{eq:ruf}.

\subsection{Scalar endpoint inequality}
We first control the exterior mass of a radial function at a prescribed level.

\begin{lemma}
\label{lem:exterior-descent}
Let $w\in H^1(\R^2)$ be nonnegative, radial and radially non-increasing, and let $\omega$ be its nonnegative, locally absolutely continuous radial representative on $(0,\infty)$. Suppose that $\omega(R)=1$ for some $R>0$. Put
\[
 h_R=\int_{\{|x|>R\}}|\nabla w|^2\,dx.
\]
If $h_R\le1$, then $h_R>0$ and
\begin{equation}  \label{eq:exterior-descent}
 R^2h_R^2e^{\frac{4\pi}{h_R}}
 \le \frac{32\pi}{1-5e^{-2}}\int_{|x|>R}w^2\,dx.
\end{equation}
\end{lemma}

\begin{proof}
By assumption, $\omega(R)=1$. Define
\[
 \zeta(t):=\omega(Re^t),\quad t\ge0,
 \qquad
 \nu:=\int_0^\infty |\zeta'(t)|^2\,dt.
\]
By the chain rule and the change of variables $r=Re^t$, we have
\begin{align} \label{eq:descent-mass-log}
 \nu
 =\int_R^\infty r|\omega'(r)|^2\,dr=\frac{h_R}{2\pi},
 \qquad
 \int_{|x|>R}w(x)^2\,dx
 =2\pi R^2\int_0^\infty \zeta(t)^2e^{2t}\,dt.
\end{align}
In particular, $\zeta(0)=1$.  If $\nu=0$, then $\zeta=1$ a.e.\ in
$(0,\infty)$, and~\eqref{eq:descent-mass-log} contradicts
$w\in L^2(\R^2)$.  Hence $\nu>0$, equivalently $h_R>0$.

If $h_R\le1$, then $0<\nu\le1/(2\pi)<1$. For every $t \ge 0$, by the absolute continuity of $\zeta$ and the Cauchy--Schwarz inequality, we obtain 
\begin{equation}
 \zeta(t)\ge \zeta(0)-\int_0^t|\zeta'(s)|\,ds
 \ge1-\sqrt{\nu t}.
 \label{eq:descent-CS}
\end{equation}
For $0\le\tau\le1$, set $t=\nu^{-1}-\tau$. Then $t\ge0$ because $\nu<1$.  Since $0\le \nu\tau\le1$, by estimate~\eqref{eq:descent-CS}, we find 
$\zeta(\nu^{-1}-\tau)\ge1-\sqrt{1-\nu\tau}\ge\frac{\nu\tau}{2}$.
Restricting~\eqref{eq:descent-mass-log} to
$[\nu^{-1}-1,\nu^{-1}]$ and substituting
$t=\nu^{-1}-\tau$, we obtain
\begin{align*}
 \int_{|x|>R}w^2\,dx
 \ge2\pi R^2\int_{\nu^{-1}-1}^{\nu^{-1}}\zeta(t)^2e^{2t}\,dt \ge\frac{\pi}{2}R^2\nu^2e^{2/\nu}
       \int_0^1\tau^2e^{-2\tau}\,d\tau
       =\frac{1-5e^{-2}}{32\pi}
        R^2h_R^2e^{\frac{4\pi}{h_R}}.
\end{align*}
The last equality is due to $h_R=2\pi\nu$ and 
$\int_0^1\tau^2e^{-2\tau}\,d\tau=(1-5e^{-2})/4$. 
This proves~\eqref{eq:exterior-descent}.
\end{proof}

We shall also use the following consequence of the classical Moser--Trudinger inequality, which asserts that there is a universal constant $C_{\rm MT}>0$ such that
\begin{equation}  \label{eq:ball-Moser-scaled}
 \int_{B_R}e^{4\pi |w|^2}\,dx\le C_{\rm MT}R^2
 \quad\text{for all } w \in H_0^1(B_R),\quad
 \|\nabla w \|_{L^2(B_R)} \leq 1.
\end{equation}
Indeed, the case $R=1$ is~\eqref{eq:classical-mt}, and the general case follows by dilation.

\begin{proof}[\textbf{Proof of Proposition~\ref{prop:scalar-exact-growth}}]
We divide the proof into the rearrangement reduction and the
high-amplitude estimate.

\smallskip
\noindent\textbf{Step 1:} We will use the rearrangement argument.
By the diamagnetic inequality, we have $|u|\in H^1(\R^2)$ and
$\|\nabla|u|\|_{L^2} \le\|\nabla u\|_{L^2}$.  Let $w=(|u|)^*$ be the symmetric
decreasing rearrangement of $|u|$.  By equimeasurability (see \cite[Subsection 3.3]{LiebLoss01}) and the
P\'olya--Szeg\H{o} inequality (see \cite{PolSze51}), we have
$\|w\|_{L^2}=\|u\|_{L^2}$, $\|\nabla w\|_{L^2} \leq \|\nabla u \|_{L^2} \leq 1$ and 
 $\int_{\R^2}\Theta_{4\pi}(w^2)\,dx=\int_{\R^2}\Theta_{4\pi}(|u|^2)\,dx$.
Therefore, it suffices to
prove~\eqref{eq:scalar-exact-growth} for $w$.

Write $w(x)=\omega(r)$ with $r=|x|$.  Since $w\in L^2(\R^2)$ and
$\omega$ is non-increasing, we have $\omega(r)\to0$ as $r \to \infty$.  
By using the power series expansion of $\Theta_{4\pi}$, we have, for $0 \leq s \leq 1$,
\begin{equation}
 \Theta_{4\pi}(s)\le C_\Theta s
 \quad \text{ where } C_\Theta:=\sum_{j=2}^\infty\frac{(4\pi)^j}{j!}
 =e^{4\pi}-1-4\pi.
 \label{eq:Theta-small-linear}
\end{equation}
If $w\le1$ a.e.\ in $\R^2$, then  \eqref{eq:scalar-exact-growth} follows from
\eqref{eq:Theta-small-linear}.  Therefore, we may 
assume that $\{w>1\}$ has positive measure and define
$R:=\sup\{r>0:\omega(r)>1\}$. 
Then $R>0$.  Also from the definition of $R$, we have $R \leq \frac{\| w\|_{L^2}}{\sqrt{\pi}}<\infty$.  We infer from the local continuity and the monotonicity of $\omega$ that
\begin{equation}  \label{eq:unit-superlevel-ball}
 \omega(R)=1,\quad \omega(r)>1 \text{ for } 0<r<R,\quad
 \omega(r) \leq 1 \text{ for } r>R.
\end{equation}

\smallskip
\noindent\textbf{Step 2:} We will prove
$\int_{\R^2}\Theta_{4\pi}(w^2)\,dx\le C\|w\|_{L^2(\R^2)}^2$.
Set $h_R:=\int_{|x|>R}|\nabla w|^2\,dx$.
By Lemma~\ref{lem:exterior-descent}, $h_R>0$ and $h_R \leq \|\nabla w\|_{L^2}^2 \leq 1$.  

We find that $h_R<1$. Indeed, if $h_R=1$, then
$\|\nabla w\|_{L^2}^2=1$ and $\int_{B_R}|\nabla w|^2\,dx=0$. Hence $w$ is constant a.e.\ on the connected ball $B_R$. Since $\omega(R)=1$, we derive that $w=1$ a.e. in $B_R$, which contradicts \eqref{eq:unit-superlevel-ball}. 

Put
\[
 v:=(w-1)_+,\qquad l_R:=\sqrt{1-h_R} > 0,\qquad \psi:=\frac{v}{l_R}.
\]
Then $v \in H_0^1(B_R)$. Moreover,
\begin{equation}
 \int_{B_R}|\nabla v|^2\,dx
 =\int_{B_R\cap\{w>1\}}|\nabla w|^2\,dx
 \le\|\nabla w\|_{L^2}^2-h_R\le1-h_R=l_R^2.
\end{equation}
Thus $\psi \in H_0^1(B_R)$, $\psi \ge0$, and $\|\nabla \psi \|_{L^2} \leq 1$.

We next prove the pointwise estimate that couples the core to the
exterior energy.  Since $l_R^2=1-h_R$, we have, for $\ell \ge0$,
\begin{equation}  \label{eq:shift-square-identity}
 (1+l_R \ell )^2 - \ell^2-\frac{1}{h_R}
 =-h_R\left(\ell -\frac{l_R}{h_R}\right)^2.
\end{equation}

We claim that, for $\ell \geq 0$,
\begin{equation} \label{eq:shifted-core-pointwise}
 \frac{e^{4\pi(1+l_R \ell )^2}}{(1+l_R \ell)^2}
 \le 16 h_R^2e^{\frac{4\pi}{h_R}}e^{4\pi \ell^2}.
\end{equation}
By~\eqref{eq:shift-square-identity}, estimate~\eqref{eq:shifted-core-pointwise} is equivalent to
\[
 S(h_R,\ell):=\frac{e^{-4\pi h_R(\ell-l_R/h_R)^2}}{h_R^2(1+l_R\ell)^2}\leq 16.
\]
If $h_R\ge1/2$, then $S(h_R,\ell)\le h_R^{-2}\le4$.  Now suppose that
$0<h_R<1/2$, so $l_R^2>1/2$.  When $\ell \ge l_R/(2h_R)$,
$h_R^2(1+l_R \ell)^2 \geq h_R^2 l_R^2 \ell^2\ge\frac{l_R^4}{4}\geq \frac{1}{16}$
and therefore $S(h_R,\ell) \leq 16$.  When $0 \leq \ell <l_R /(2h_R)$,
$|\ell-l_R/h_R|\ge l_R/(2h_R)$, whence
$S(h_R,\ell)\le h_R^{-2}e^{-\pi l_R^2/h_R}
 \leq h_R^{-2}e^{-\pi/(2h_R)}
 \leq 4e^{-\pi}<1$. For the last inequality we used that
$h\mapsto h^{-2}e^{-\pi/(2h)}$ is increasing on $(0,1/2]$.
Therefore, we get \eqref{eq:shifted-core-pointwise}.

By~\eqref{eq:unit-superlevel-ball}, $w=1+v=1+l_R \psi$ a.e.\ in
$B_R$.  Since $w\ge1$ there and
$\Theta_{4\pi}(s)\le e^{4\pi s}/s$ for $s>0$, we infer from inequalities \eqref{eq:shifted-core-pointwise} and
\eqref{eq:ball-Moser-scaled} that
\begin{align}
 \int_{B_R}\Theta_{4\pi}(w^2)\,dx
 \le \int_{B_R}\frac{e^{4\pi w^2}}{w^2}\,dx \le 16h_R^2e^{\frac{4\pi}{h_R}}
       \int_{B_R}e^{4\pi \psi^2}\,dx \le 16C_{\rm MT}R^2h_R^2e^{\frac{4\pi}{h_R}}
 \le C_3\int_{|x|>R}w^2\,dx,
 \label{eq:core-by-tail-mass}
\end{align}
where in the last inequality, we have used  Lemma~\ref{lem:exterior-descent} and $C_3=C_{\rm MT}\frac{2^9\pi}{1-5e^{-2}}$.

On $B_R^c$, $w\le1$ by~\eqref{eq:unit-superlevel-ball}, hence thanks to \eqref{eq:Theta-small-linear}, we have
$\int_{B_R^c}\Theta_{4\pi}(w^2)\,dx
 \le C_\Theta\int_{B_R^c}w^2\,dx$.
This, together with~\eqref{eq:core-by-tail-mass} and $\int_{B_R^c}|w|^2\,dx\le\|w\|_{L^2}^2=\|u\|_{L^2}^2$, implies \eqref{eq:scalar-exact-growth}. The proof is complete.
\end{proof}

\subsection{Comparison with Ibrahim--Masmoudi--Nakanishi's inequality and  Ruf's inequality}
Put
\[
\mathcal G(s):=\frac{e^{4\pi s}-1}{(1+\sqrt{s})^2}, \quad \text{for } s\ge0.
\]

\begin{lemma}
We have
\begin{equation}  \label{eq:Theta-IMN-comparison}
 \mathcal G(s)\leq \Theta_{4\pi}(s) \leq 16\pi \mathcal G(s),\quad s \geq 0.
\end{equation}
\end{lemma}
\begin{proof}
For $s>0$, write
$\tau=\sqrt{s}$ and $t=4\pi s$.  Direct calculation gives
\begin{equation}
 \frac{\Theta_{4\pi}(s)}{\mathcal G(s)}
 =\frac{(1+\tau)^2}{\tau^2}
  \left(1-\frac{t}{e^t-1}\right).
 \label{eq:Theta-calG-ratio}
\end{equation}
Since $e^t-1\ge t+t^2/2$, $1-\frac{t}{e^t-1}\ge\frac{t}{2+t}$. 
Hence, by~\eqref{eq:Theta-calG-ratio},
$
 \frac{\Theta_{4\pi}(s)}{\mathcal G(s)}
 \geq \frac{4\pi(1+\tau)^2}{2+4\pi \tau^2}
 \geq 1.
$

If $0<t<1$, then $e^t\le(1-t)^{-1}$ and consequently
$1-t/(e^t-1)\le t$.  The same conclusion is immediate at $t=1$.
Thus, when $0<t\le1$,
$
 \frac{\Theta_{4\pi}(s)}{\mathcal G(s)}
 \leq4\pi(1+\tau)^2\leq16\pi.
$
If $t\ge1$, then by \eqref{eq:Theta-calG-ratio}, we obtain
\[
 \frac{\Theta_{4\pi}(s)}{\mathcal G(s)}
 \leq \left (1+\frac{1}{\tau} \right )^2
 \leq(1+\sqrt{4\pi})^2<16\pi.
\]
The inequalities also hold for $s=0$. The proof is complete. 
\end{proof}

\begin{corollary}
\label{cor:ruf-from-exact-growth}~

(i) Estimate~\eqref{eq:scalar-exact-growth} is equivalent
to \eqref{eq:IMN-exact-growth-intro}.

(ii) Estimate~\eqref{eq:scalar-exact-growth} implies \eqref{eq:ruf}.
\end{corollary}

\begin{proof}
\textit{(i)} Thanks to \eqref{eq:Theta-IMN-comparison},   \eqref{eq:scalar-exact-growth} and \eqref{eq:IMN-exact-growth-intro} are  equivalent.

\textit{(ii)} Assuming \eqref{eq:scalar-exact-growth}, we will prove \eqref{eq:ruf}. 
Let $w \in H^1(\R^2)$ such that $\| \nabla w \|_{L^2} \leq 1$. Let $m\geq1$ be an integer. Considering the term with index $j=m+1$ in the power series of $\Theta_{4\pi}$ and using~\eqref{eq:scalar-exact-growth}, we obtain
\[
 \frac{(4\pi)^{m+1}}{(m+1)!}
 \int_{\R^2}|w|^{2m}\,dx
 \le \int_{\R^2}\Theta_{4\pi}(|w|^2)\,dx
 \le C_{\rm sc}\|w\|_{L^2}^2.
\]
Taking the $m$th root and using
$(m+1)! \leq (m+1)^{m+1}$, $m+1 \leq 2m$, and
$(m+1)^{1/m} \leq 2$, we obtain
\begin{align*}
 \big\||w|^2\big\|_{L^m}
 \leq \left(\frac{C_{\rm sc}(m+1)!}{(4\pi)^{m+1}}\right)^{\frac{1}{m}}
 \|w\|_{L^2}^{\frac{2}{m}} \leq
 \frac{C_{\rm sc}^{\frac{1}{m}}}{(4\pi)^{1+\frac{1}{m}}}
 (m+1)^{1+\frac{1}{m}}\|w\|_{L^2}^{\frac{2}{m}}
 \le C_4 m \|w\|_{L^2}^{\frac{2}{m}}.
\end{align*}
Here $C_4$ is universal because both
$C_{\rm sc}^{1/m}$ and $(4\pi)^{-1/m}$ are uniformly bounded for $m\ge1$.  Thus
\begin{equation}  \label{eq:integer-exact-growth-moments}
 \big\||w|^2\big\|_{L^m}
 \le C_4 m \,\| w \|_{L^2}^{\frac{2}{m}},
 \quad m \in \N.
\end{equation}

For an arbitrary real exponent $p\ge1$, choose $m \in \N$ such that $m \le p\le m+1$, and 
$\eta\in[0,1]$ so that
$p^{-1}=\eta m^{-1}+(1-\eta)(m+1)^{-1}$.  Then by H\"older's inequality and~\eqref{eq:integer-exact-growth-moments}, we have
\begin{align*}
 \big\||w|^2\big\|_{L^p} \le \big\||w|^2\big\|_{L^m}^\eta
       \big\||w|^2\big\|_{L^{m+1}}^{1-\eta} \le C_4 m^\eta(m+1)^{1-\eta}
       \|w\|_{L^2}^{\frac{2\eta}{m}+\frac{2(1-\eta)}{(m+1)}}
 \le C_5 p\,\|w\|_{L^2}^{\frac{2}{p}}.
\end{align*}
Here the mass exponent equals $2/p$ by the definition of $\eta$, and
$m^\eta(m+1)^{1-\eta}\leq m+ 1 \leq 2p$.  Therefore we have proved that for any $p \geq 1$,
\begin{equation}  \label{eq:exact-growth-moments}
 \big\||w|^2\big\|_{L^p}
 \le C_5 p\,\|w\|_{L^2}^{\frac{2}{p}}
 \quad \text{for all } w \in H^1(\R^2) \text{ with } \|\nabla w\|_{L^2} \leq 1.
\end{equation}

Now let $u\in H^1(\R^2)$ satisfy $\|u\|_{H^1} \leq 1$, and put
$a_1:=\|u\|_{L^2}^2$ and $a_2:=\|\nabla u\|_{L^2}^2$. Then $a_1+a_2\le1$.
If $a_1=0$, then $u=0$.  If $a_1\ge1/2$, then
$w:=\sqrt2u$ satisfies
$\|\nabla w\|_{L^2}^2=2a_2\le2(1-a_1)\le1$.
By the elementary bound
$e^{2\pi s}-1\le C\bigl(s+\Theta_{4\pi}(s)\bigr)$ for $s\ge0$
and Proposition~\ref{prop:scalar-exact-growth}, we derive
\[
 \int_{\R^2}(e^{4\pi|u|^2}-1)\,dx
 \leq C\left(\|w\|_{L^2}^2+
       \int_{\R^2}\Theta_{4\pi}(|w|^2)\,dx\right)\le C.
\]

It remains to consider $0<a_1<1/2$.  Set
$w:=\frac{u}{\sqrt{1-a_1}}$. 
Then
$\|\nabla w\|_{L^2}^2=a_2/(1-a_1)\le1$ and
$\|w\|_{L^2}^2=a_1/(1-a_1)$.  On
$A:=\{x\in\R^2:|u(x)|\ge1\}$ one has $|w|\ge1$.  Since
$e^{4\pi s}-1-4\pi s
 \ge \bigl(1-(1+4\pi)e^{-4\pi}\bigr)e^{4\pi s}$ for $s\ge1$,
by Proposition~\ref{prop:scalar-exact-growth}, we obtain
\begin{equation}
 \int_A\frac{e^{4\pi|w|^2}}{|w|^2}\,dx
 \le C\int_{\R^2}\Theta_{4\pi}(|w|^2)\,dx
 \le C\frac{a_1}{1-a_1}.
 \label{eq:ruf-high-exact}
\end{equation}
Put $p:=(1-a_1)/a_1>1$.  Thus
$\|w\|_{L^2}^2=1/p$ and $p a_1=1-a_1$.  By the pointwise identity
$e^{4\pi|u|^2}
 =\left(\frac{e^{4\pi|w|^2}}{|w|^2}\right)^{1-a_1}
  \left(|w|^{2p}\right)^{a_1}$,  H\"older's inequality, 
\eqref{eq:ruf-high-exact} and~\eqref{eq:exact-growth-moments}, we derive
\begin{equation} \label{integral-A}
 \int_A e^{4\pi|u|^2}\,dx
 \le \left(\frac{C}p\right)^{1-a_1}
       \left(\int_A|w|^{2p}\,dx\right)^{a_1} \le \left(\frac{C}p\right)^{1-a_1}
       \left(C_5 p\Big(\frac{1}{p}\Big)^{\frac{1}{p}}\right)^{p a_1}
 \le C p^{-a_1}\le C.
\end{equation}
Finally, on $A^c$,
$e^{4\pi|u|^2}-1 \leq (e^{4\pi}-1)|u|^2$, hence
\begin{equation} \label{eq:integral-A-complement} \int_{A^c}\bigl(e^{4\pi|u|^2}-1\bigr)\,dx \leq (e^{4\pi}-1) \int_{A^c}|u|^2\,dx \leq e^{4\pi}-1.
\end{equation}
Combining \eqref{integral-A} and \eqref{eq:integral-A-complement} leads to \eqref{eq:ruf}.
\end{proof}

\begin{remark} \label{rem:scalar-comparisons}
For completeness, we provide the scaling argument that connects the scalar estimate under the homogeneous gradient constraint to the corresponding inequality under the full $H^1$-norm constraint. For $N=1$, estimate~\eqref{eq:rank-one-defect} becomes
\begin{equation}  \label{eq:inhomogeneous-scalar-exact}
 \int_{\R^2}\Theta_{4\pi}(|u|^2)\,dx
 \le C\|u\|_{L^2}^2
 \quad \text{for all } u \in H^1(\R^2) \text{ with } \|u\|_{H^1}=1.
\end{equation}
We observe that inequality \eqref{eq:scalar-exact-growth} implies  \eqref{eq:inhomogeneous-scalar-exact} because $\|u\|_{H^1}=1$ entails
$\|\nabla u\|_{L^2} \le 1$.  Conversely, assume
\eqref{eq:inhomogeneous-scalar-exact} and let $u\in H^1(\R^2)$ with $\ell_1:=\|u\|_{L^2}^2$ and
$\ell_2:=\|\nabla u\|_{L^2}^2\leq 1$.
The case $\ell_1=0$ is trivial.  Now we consider the case $\ell_1>0$. If $\ell_2<1$, then we set
\[
 \lambda^2:=\frac{\ell_1}{1-\ell_2},
 \quad v(x):=u(\lambda x).
\]
The two-dimensional scaling gives
$\|\nabla v\|_{L^2}^2=\ell_2$ and
$\|v\|_{L^2}^2=\lambda^{-2}\ell_1=1-\ell_2$, and hence
$\|v\|_{H^1}=1$. Applying
\eqref{eq:inhomogeneous-scalar-exact} to $v$ and changing variables
yields
$
 \lambda^{-2}\int_{\R^2}\Theta_{4\pi}(|u|^2)\,dx
 \le C\lambda^{-2}\ell_1,
$
which is the homogeneous estimate.  If $\ell_2=1$, then we apply the preceding
argument to $u_\delta:=\sqrt{1-\delta}\,u$, $0<\delta<1$, and let
$\delta\downarrow0$.  Since $\Theta_{4\pi}$ is increasing, the monotone convergence theorem yields the same conclusion. Hence the homogeneous and
inhomogeneous estimates are equivalent.
\end{remark}

\section{Auxiliary semigroup estimates}
\label{sec:auxiliary-estimates}

We collect elementary estimates used in the semigroup  and high-density arguments.

\begin{lemma}
\label{lem:scalar-identity}
For every $\lambda>0$,
$$
 \int_0^\infty\frac{(1-e^{-t\lambda})^2}{t^2}\,dt
 =2(\ln2)\lambda.
$$
\end{lemma}

\begin{proof}
The change of variables $s=t\lambda$ gives
\[
 \int_0^\infty\frac{(1-e^{-t\lambda})^2}{t^2}\,dt
 =\lambda\int_0^\infty\frac{(1-e^{-s})^2}{s^2}\,ds.
\]
Since $(1-e^{-s})^2/s\to0$ as $s\downarrow0$ and as
$s\to\infty$, integration by parts gives
\[
 \int_0^\infty\frac{(1-e^{-s})^2}{s^2}\,ds
 =2\int_0^\infty\frac{e^{-s}-e^{-2s}}s\,ds
 =2\ln2,
\]
where the last equality follows from Frullani's identity
$\int_0^\infty(e^{-as}-e^{-bs})s^{-1}\,ds=\ln(b/a)$
for $a,b>0$, applied with $a=1$ and $b=2$.
\end{proof}

\begin{lemma}
\label{lem:J-lower}
There is a universal constant $c>0$ such that for every  $\sigma\ge0$,
\begin{equation} \label{append:est-1}
 \int_0^\sigma e^s(\sqrt \sigma-\sqrt s)^2\,ds
 \ge c \left(\frac{e^\sigma}{1+\sigma}-1\right).
\end{equation}
\end{lemma}

\begin{proof}
First suppose that  $\sigma\ge1$ and put $r=\sigma-s$. For
$s\in[\sigma-1,\sigma]$, we have $r\in[0,1]$ and
 $\sqrt \sigma-\sqrt{\sigma-r}
 =\frac{r}{\sqrt \sigma+\sqrt{\sigma-r}}
 \ge\frac{r}{2\sqrt \sigma}$.
Set $c_L:=\frac{1}{4}\int_0^1r^2e^{-r}\,dr=(2-5/e)/4>0$. Consequently,
\begin{equation}
 \int_0^\sigma e^s(\sqrt \sigma-\sqrt s)^2\,ds
 \ge c_L\frac{e^\sigma}{\sigma}
 \ge c_L\left(\frac{e^\sigma}{1+\sigma}-1\right).
\end{equation}

If  $0\le \sigma\le1$, then on  $0<s<\sigma/4$, we have
$\sqrt \sigma-\sqrt s\ge\sqrt \sigma/2$, and hence
 $\int_0^\sigma e^s(\sqrt \sigma-\sqrt s)^2\,ds\ge\frac{\sigma^2}{16}$.
For $b(\sigma):=e^\sigma/(1+\sigma)$, one has  $b(0)=1$, $b'(0)=0$, and
$ b''(\sigma)=\frac{e^\sigma(1+\sigma^2)}{(1+\sigma)^3}\le2e$ for $0 \leq \sigma\leq 1$.
Therefore, by using Taylor's formula, we deduce  $b(\sigma)-1\le e\sigma^2$, so
$$
 \int_0^\sigma e^s(\sqrt \sigma-\sqrt s)^2\,ds
 \ge\frac{1}{16e}\left(\frac{e^\sigma}{1+\sigma}-1\right).
$$
Combining the two cases yields \eqref{append:est-1} with
$c:=\min\{(2-5/e)/4,1/(16e)\}>0$.
\end{proof}
Define
\begin{equation} \label{def:Phi}
 \Phi(\tau):=\inf_{t>0}
 \left[e^{2t\tau}\ln\left(1+\frac{1}{2t}\right)\right],
 \quad\tau>0.
\end{equation}

\begin{lemma}
\label{lem:optimized-Phi}
For every $\tau\ge e^2$,
\begin{equation}
 \Phi(\tau)
 \le\left(1+\frac{2}{\ln\tau}\right)
 \bigl(\ln\tau+\ln\ln\tau+\ln2\bigr).
\label{eq:Phi-bound}
\end{equation}
In particular, $\limsup_{\tau\to\infty}\frac{\Phi(\tau)}{\ln\tau}\le1$.
\end{lemma}

\begin{proof}
For $\tau\ge e^2$, choose $t=(2\tau\ln\tau)^{-1}$. Then
$2t\tau=1/\ln\tau\le1/2$, and 
$e^{2t\tau}\le1+\frac{2}{\ln\tau}$.
Moreover, $1/(2t)=\tau\ln\tau\ge1$, so
\[
 \ln\left(1+\frac{1}{2t}\right)
 \le\ln(2\tau\ln\tau)
 =\ln\tau+\ln\ln\tau+\ln2.
\]
This and $e^{2t\tau} \le 1+\frac{2}{\ln\tau}$ imply \eqref{eq:Phi-bound}. Then we derive $\limsup_{\tau\to\infty}\frac{\Phi(\tau)}{\ln\tau}\le1$ from \eqref{eq:Phi-bound}.
\end{proof}

 \noindent \textbf{Data availability.} This research has no associated data. \medskip

\noindent \textbf{Conflict of interest.} The authors declare that they have no conflicts of interest. \medskip

\noindent \textbf{AI assistance statement.} 
The authors used OpenAI ChatGPT/Codex to assist with calculations and polishing the writing. The authors retain full responsibility for the main ideas, mathematical verification, and all final review.

\end{document}